\documentclass[12pt]{amsart}
\usepackage{amsfonts}
\usepackage{fullpage}
\usepackage{amsmath,amsthm,bbm}
\usepackage{latexsym}
\usepackage{array}
\usepackage{dsfont}
\usepackage{amssymb}
\usepackage{graphics}
\usepackage{enumerate}
\usepackage[english]{babel}

\usepackage{pict2e}
\usepackage{xcolor}
\usepackage{pgf, tikz}

\usepackage{color}
\definecolor{marin}{rgb}   {0.,   0.3,   0.7} 
\definecolor{rouge}{rgb}   {0.8,   0.,   0.} 
\definecolor{sepia}{rgb}   {0.8,   0.5,   0.} 
\usepackage[colorlinks,citecolor=marin,linkcolor=rouge,
            bookmarksopen,
            bookmarksnumbered
           ]{hyperref}

\newtheorem{lemma}{Lemma}[section]

\newtheorem{theorem}[lemma]{Theorem}
\newtheorem{proposition}[lemma]{Proposition}
\newtheorem{corollary}[lemma]{Corollary}
\newtheorem{remark}[lemma]{Remark}
\newtheorem{example}[lemma]{Example}

\newtheorem{notation}[lemma]{Notation}
\newtheorem{definition}[lemma]{Definition}
\newtheorem{conclusion}[lemma]{Conclusion}
\numberwithin{equation}{section}
\newcommand{\QED}{\mbox{}\hfill \raisebox{-0.2pt}{\rule{5.6pt}{6pt}\rule{0pt}{0pt}} 
          \medskip\par}             
\newenvironment{Proof}{\noindent
    \parindent=0pt\abovedisplayskip = 0.5\abovedisplayskip
    \belowdisplayskip=\abovedisplayskip{\bfseries Proof. }}{\QED}
\newenvironment{Proofof}[1]{\noindent
    \parindent=0pt\abovedisplayskip = 0.5\abovedisplayskip
    \belowdisplayskip=\abovedisplayskip{\bfseries Proof of #1. }}{\QED}

\newcommand{\eps}{\varepsilon}

\newcommand{\dd}{\mathrm{d}}

\newcommand{\Ec}{\mathcal{E}}
\newcommand{\Ecs}{\mathcal{E}^{*}}

\newcommand{\rX}{\mathrm{X}}
\newcommand{\rD}{\mathrm{D}}
\newcommand{\rU}{\mathrm{U}}

\author[R. Horsin]{Romain Horsin}
\address{INRIA-Rennes Bretagne Atlantique and IRMAR (UMR 6625) Universit\'e de Rennes I} 
\email{romainhm24@gmail.com}
\urladdr{}

\thanks{This work was partially supported by the ERC starting grant GEOPARDI No. 279389}         

\title[On time-discretization of the 2D Euler equation by a symplectic Crouch-Grossman integrator]
{On time-discretization of the 2D Euler equation by a symplectic Crouch-Grossman integrator}
 
\begin{document}

\begin{abstract}
We consider time discretizations of the two-dimensional Euler equation written in vorticity form. The discretization method uses a Crouch-Grossman integrator that proceeds in two stages: first freezing the velocity vector field at the beginning of the time step, and then solve the resulting elementary transport equation by using a symplectic integrator to discretize in time the flow of the associated Hamiltonian differential equation. We prove that these schemes yield order one approximations of the exact solutions, and provide error estimates in Sobolev norms.
\end{abstract}


\maketitle

\section{Introduction}

In this paper we consider time discretizations of the two-dimensional Euler equation for perfect incompressible fluids, written in vorticity form, and with periodic boundary conditions. The discretization in time will use Crouch-Grossman integrator. The reader will find details on these integrators in \cite{C-G}, but essentially we mean by using this denomination that one iteration in time requires two stages. Indeed, as we consider a transport equation associated with a non-autonomous Hamiltonian vector field, the first stage of the time-scheme is to freeze this Hamiltonian at the beginning of the time step. In a second time, we discretize in time the characteristics associated with the frozen Hamiltonian. For that we shall use the celebrated implicit midpoint method (see \cite{Erwan,HLW}), to preserve moreover the symplectic structure. Our purpose is to prove the convergence of the time scheme.\\
The vorticity formulation of the two-dimensional Euler equation in a periodic box reads
\begin{equation}
\label{euler2d}
\left\{
\begin{split}
&\partial_{t}\omega - \rU(\omega)\cdot\nabla \omega =0, \\
&\omega(0,x)=\omega_{0}(x),
\end{split}
\right.
\end{equation}
where $\omega(t,x) \in \mathbbm{R},$ with $t\in\mathbbm{R}_{+},$ and $x\in\mathbbm{T}^{2},$ the two-dimensional torus defined by
$$\mathbbm{T}^{2}=(\mathbbm{R}/2\pi \mathbbm{Z})^{2}.$$
The divergence-free velocity vector field $\rU(\omega)$ is given by the formula
$$\rU(\omega)= J \nabla \Delta^{-1} \omega$$
using the canonical symplectic matrix 
$$J= \left( \begin{array}{ccc}
0 & 1 \\
-1 & 0 \end{array} \right).$$
$\Delta^{-1}$ stands for the inverse of the Laplace operator on functions with average $0$ on $\mathbbm{T}^{2}$ (see appendix A), and $\nabla$ is the two-dimensional gradient operator.\\
Note that the vorticity form \eqref{euler2d} is formally equivalent to standard formulations of the Euler equation that may be found in the literature (see \cite{Constantin}).

In this paper we consider the time-discretization of \eqref{euler2d} by a Crouch-Grossman integrator (see \cite{C-G}), that proceeds in two stages. First, freezing the velocity vector field at the beginning of the time step, which gives the Hamiltonian transport equation
\begin{equation}
\label{eq-frozen}
\left\{
\begin{split}
&\partial_{t} f  - J\nabla \psi \cdot \nabla f= 0\\
& \Delta \psi=\omega_{0} 
\end{split}
\right.
\end{equation}
with initial data $\omega_{0}.$ \\
The second step is to discretize the flow of the vector field $J\nabla\psi$ by the implicit midpoint method. More precisely, we define $\Phi_{t}(x)$ as the unique solution of the implicit equation
$$\Phi_{t}(x)=x+tJ\nabla\psi\left(\frac{x+\Phi_{t}(x)}{2}\right),$$
and if $t$ is small enough, $\omega_{0}\circ \Phi_{t}(x)$ should be an approximation of the solution $f(t,x)$ of \eqref{eq-frozen}. The existence of an unique solution $\Phi_{t}(x)$ of the implicit equation requires regularity on $\psi,$ and will be carefully justified in Proposition \ref{Mid-topo} below.\\
From then we define the semi-discrete operator $\mathcal{S}_{t}$ by
\begin{equation}
\label{SD-OP}
\left\{
\begin{split}
&\mathcal{S}_{t}(\omega_{0})(x) = \omega_{0}\left(\Phi_{t}(x)\right),\\
&\Phi_{t}(x)=x+tJ\nabla\psi\left(\frac{x+\Phi_{t}(x)}{2}\right),\\
&\Delta\psi =\omega_{0}.
\end{split}
\right.
\end{equation}
If $\tau\in]0,1[$ is the time-step, we define a sequence $(\omega_{n})_{n\in\mathbbm{N}}$ by
\begin{equation}
\label{omega-n}
\left\{
\begin{split}
&\omega_{n}(x)=\mathcal{S}_{\tau}(\omega_{n-1}) = \mathcal{S}_{\tau}^{n}(\omega_{0})\\
&\mathcal{S}_{\tau}^{0}(\omega_{0})=\omega_{0}.
\end{split}
\right.
\end{equation}

The main result of the paper is that, if $\tau$ is small enough, $\omega_{n}(x)=\mathcal{S}_{t}^{n}(\omega_{0})(x)$ is an order one approximation in $H^{s}$ (Sobolev space, see definition \ref{Sobo} below) of $\omega(t,x)$ at time $t_{n}=n\tau,$ where $\omega$ solves \eqref{euler2d}.

To the best of our knowledge, there is currently no paper available in the literature in which such time-schemes are considered for the Euler equation. Another approach, based on optimal transport, has however been recently used in \cite{Geodesic1} (see also \cite{Geodesic2}) to discretize the Euler equation on a general domain $\Omega\in\mathbbm{R}^{d}$ with Lipschitz boundary. It is based on the fact that the expression of Euler's equation in Lagrangian coordinates may be seen as the equation of geodesics in the group made of the diffeomorphisms on $\Omega$ that preserve the restriction (to $\Omega$) of the Lebesgue measure, as noticed in \cite{Arnold}. From then the authors discretize in space by considering the equation of geodesics in a finite dimensional subspace of $L^{2}(\Omega,\mathbbm{R}^{d})$ (see also \cite{Brenier}). This approximate geodesics equation, which is a Hamiltonian ODE, is then discretized in time by using Euler's symplectic integrator. It is proven in \cite{Geodesic1} that these schemes yield approximations in space and time of strong solutions of the Euler equation.\\
The issue of space discretization has also recently been handled in \cite{Bardos}, where spectral methods are applied to the Burgers and Euler equations. The spectral methods can be seen as regularization procedures, as they consist in replacing the exact solutions of the Euler equations by their truncated Fourier series, where the modes of high frequencies are cut off. In practice, the method is implemented by using discrete Fourier series (see \cite{Fourier}). Difficulties arise from aliasing errors, which are overcome by the use of the $2/3$ de-asliasing method (see also \cite{Fourier}). The authors of \cite{Bardos} prove the convergence of this scheme in space, provided that the exact solutions have enough regularity.\\
The vorticity form of the 2D Euler equation may also be called the Guiding Center Model, which is also used to to describe the evolution of the charge density in a highly magnetized plasma in the transverse plane of a tokamak (see \cite{CG1} and \cite{CG2}). Papers \cite{CG1} and \cite{CG2} investigate full-discretizations of the Guiding Center Model by, respectively, forward semi-Lagrangian methods and backward semi-Lagrangian methods. See also \cite{CG4} and \cite{CG3}.

\section{Main result}

\subsection{Notations}

We shall consider functions defined on the two-dimensional torus, also seen as periodic functions on $\mathbbm{R}^{2},$ with period $2\pi$ in each variables. Therefore a function $f$ on $\mathbbm{T}^{2}$ will be usually written as
$$f(x)=f(x_1,x_2), \quad \mbox{with} \quad x=(x_1,x_2)\in [0,2\pi]\times [0,2\pi].$$
We will write
$$\langle x \rangle =\left(1+x_1^{2}+x_2^{2}\right)^{1/2}.$$
The notation $|\cdot|$ will in general refer to any norm on $\mathbbm{R}$ or $\mathbbm{R}^{2}.$ In the case of a two-dimensional integer $\alpha=(\alpha_{1},\alpha_{2})\in\mathbbm{N}^{2},$ we will write
$$|\alpha| =\alpha_{1}+\alpha_{2}.$$
For functions $f:\mathbbm{T}^{2} \to \mathbbm{C}$ and integers $\alpha=(\alpha_{1},\alpha_{2})\in\mathbbm{N}^{2},$ we will use the notation
$$\partial_{x}^{\alpha} f=\partial_{x_1}^{\alpha_{1}} \partial_{x_2}^{\alpha_{2}}f.$$
The operators $\nabla,$ $\Delta$ and $\nabla \cdot$ are defined by
$$\nabla f=(\partial_{x_1}f,\partial_{x_2}f)^\top, \quad \Delta =\partial_{x_1}^{2}f+\partial_{x_2}^{2}f, \quad \nabla \cdot \rX = \partial_{x_1} \rX_{1} + \partial_{x_2}\rX_{2},$$
where  $\rX$ is a two-dimensional vector field $\rX=\left(\rX_{1},\rX_{2}\right):\mathbbm{T}^{2}\to \mathbbm{R}^{2}.$ In particular, $\nabla^{2}f$ will be the Hessian matrix of $f.$\\
We shall also write 
$$\partial_{x}^{\alpha}\rX=\left(\partial_{x}^{\alpha}\rX_{1},\partial_{x}^{\alpha}\rX_{2}\right)^{\top}.$$
The differential of $\rX$ shall be denoted by $\rD_{x}\rX,$ and is classically defined by the formula
$$\rD_{x}\rX = \left(\partial_{x}^{(1,0)} \rX, \partial_{x}^{(0,1)} \rX\right).$$

\subsection{Functional framework}

For $p\in[1,+\infty[,$ $\ell^{p}(\mathbbm{Z}^{2})$ and $L^{p}(\mathbbm{T}^{2})$ are the classical Lebesgue spaces on $\mathbbm{Z}^{2}$ and $\mathbbm{T}^{2},$ respectively equipped with the norms
$$\left\| (u_{k})_{k\in\mathbbm{Z}^{2}} \right\|_{\ell^{p}(\mathbbm{Z}^{2})} = \left(\sum_{k\in\mathbbm{Z}^{2}} |u_{k}|^{p}\right)^{1/p} \quad \mbox{and} \quad \left\| f \right\|_{L^{p}(\mathbbm{T}^{2})} = \left( \int_{\mathbbm{T}^{2}} |f(x)|^{p}\dd x\right)^{1/p},$$
where $\dd x$ stands for the normalized Lebesgue measure on $\mathbbm{T}^{2}.$ We shall use the notation
$$\langle f,g \rangle_{L^{2}(\mathbbm{T}^{2})} = \int_{\mathbbm{T}^{2}} f(x)\overline{g(x)}\dd x$$
for the usual inner product associated with the norm $\| \cdot\|_{L^{2}(\mathbbm{T}^{2})}.$\\
We will also consider the Lebesgue spaces $\ell^{\infty}(\mathbbm{Z}^{2})$ and $L^{\infty}(\mathbbm{T}^{2}),$ respectively equipped with the norms
\begin{equation}
\notag
\begin{split}
&\left\| (u_{k})_{k\in\mathbbm{Z}^{2}} \right\|_{\ell^{\infty}(\mathbbm{Z}^{2})} = \sup_{k\in\mathbbm{Z}^{2}} |u_{k}| \\
&\left\| f \right\|_{L^{\infty}(\mathbbm{T}^{2})} =\sup\left\{ M \in\mathbbm{R} \hspace{2mm}  |  \hspace{2mm} \tilde{\lambda} \left(\left\{ x \hspace{2mm} | \hspace{2mm}  |f(x)| >M\right\}\right) = 0 \right\}.
\end{split}
\end{equation}
The two-dimensional Fourier coefficients of a function $f$ on $\mathbbm{T}^{2}$ are given by
$$\hat{f}_{k}=\int_{\mathbbm{T}^{2}}f(x)e^{-ik \cdot x} \dd x, \quad k\in \mathbbm{Z}^{2},$$
where $\cdot$ is the usual inner product on $\mathbbm{R}^{2}.$\\
The Fourier series of $f$ is defined by
$$\sum_{k\in\mathbbm{Z}^{2}} \hat{f}_{k}e^{ik \cdot x}.$$
For $s\in\mathbbm{R},$ $H^{s}(\mathbbm{T}^{2})$ is the periodic Sobolev space, equipped with the norm
\begin{equation}
\label{Sobo}
\left\| f \right\|_{H^{s}(\mathbbm{T}^{2})}  = \left(\sum_{k\in\mathbbm{Z}^{2}} |\hat{f}_{k}|^{2} \langle k \rangle^{2s}\right)^{1/2} \sim \left(\sum_{0\leq |\alpha|\leq s} \left\| \partial_{x}^{\alpha} f \right\|_{L^{2}(\mathbbm{T}^{2})}^{2}\right)^{1/2}.
\end{equation}
We refer for instance to \cite{Sobolev} for basic properties of the periodic Sobolev spaces. We shall essentially use the Sobolev embeddings
\begin{equation}
\label{Sob-inj1}
H^{1}(\mathbbm{T}^{2}) \hookrightarrow L^{4}(\mathbbm{T}^{2}),
\end{equation}
and 
\begin{equation}
\label{Sob-inj2}
H^{s}(\mathbbm{T}^{2}) \hookrightarrow L^{\infty}(\mathbbm{T}^{2}), \quad \mbox{for all } s>1.
\end{equation}
The same notations $L^{p}$ and $H^{s}$ will also be used for the Lebesgue and Sobolev spaces of vector-valued functions on $\mathbbm{T}^{2}.$ We shall use in addition the following Lemma (see Proposition 3.9 of \cite{Taylor3} for instance):

\begin{lemma}
\label{Kato}
Assume that $F:\mathcal{U} \to M_{2}(\mathbbm{R})$ is a $C^{\infty}$ map satisfying $F(0)=0,$ where $\mathcal{U}$ is an open subset of $M_{2}(\mathbbm{R})$ containing $0,$ and where $M_{2}(\mathbbm{R})$ is the space of $2\times 2$ square matrices with real coefficients. For any $s> 1$ and $A\in H^{s}(\mathbbm{T}^{2}),$ such that $A\in \mathcal{U}$ almost everywhere,
$$\left\| F(A)\right\|_{H^{s}(\mathbbm{T}^{2})} \leq C_{s}\left(\left\| A\right\|_{L^{\infty}(\mathbbm{T}^{2})}\right)\left(1 + \left\| A\right\|_{H^{s}(\mathbbm{T}^{2})}\right),$$
where $C_{s}:\mathbbm{R}_{+}\to \mathbbm{R}_{+}$ is an increasing continuous function.
\end{lemma}

The two-dimensional Euler equation is globally well-posed in these Sobolev spaces. More precisely, we have the following result (see for instance the chapter 7 of \cite{Chemin} for a  proof):

\begin{theorem}
\label{existence}
Let $s>1$ and $\omega_{0}\in H^{s}(\mathbbm{T}^{2})$ with average $0.$ There exists an unique solution $\omega(t,x)\in C^{0}(\mathbbm{R}_{+}, H^{s}(\mathbbm{T}^{2}))\cap C^{1}(\mathbbm{R}_{+}, H^{s-1}(\mathbbm{T}^{2}))$ of equation \eqref{euler2d} with initial data $\omega_{0}.$
\end{theorem}

\subsection{Statement of the main result}

Our goal is to prove the following convergence Theorem:

\begin{theorem}
\label{convergence}
Let $s\geq 6$ and $\omega_{0}\in H^{s}(\mathbbm{T}^{2})$ with average $0.$ Let $\omega(t,x)\in C^{0}\left(\mathbbm{R}_{+},H^{s}(\mathbbm{T}^{2})\right)$ be the unique solution of equation \eqref{euler2d} given by Theorem \ref{existence}, with initial data $\omega_{0}.$ For a time step $\tau\in]0,1[,$ let $\left(\omega_{n}\right)_{n\in\mathbbm{N}}$ be the sequence of functions starting from $\omega_{0}$ and defined by formula \eqref{omega-n} from iterations of the semi-discrete operator \eqref{SD-OP}. For a fixed time horizon $T>0,$ let $B=B(T)$ be such that
$$\sup_{t\in [0,T]} \left\| \omega(t)\right\|_{H^{s}(\mathbbm{T}^{2})} \leq B.$$
There exists two positive constants $R_{0}$ and $R_{1},$ and an increasing continuous function $R:\mathbbm{R}_{+} \to \mathbbm{R}_{+},$ such that, if $\tau$ satisfies the hypothesis
$$\tau <\max\left(\frac{1}{R_{0}B}, \frac{B}{TR(B)e^{R_{1}T(1+B)}}\right),$$
the semi-discrete scheme enjoys the following convergence estimate: for all $n\in\mathbbm{N}$ such that $t_{n}=n\tau\leq T,$
$$\left\| \omega_{n}- \omega(t_{n})\right\|_{H^{s-4}(\mathbbm{T}^{2})}\leq \tau t_{n} R(B) e^{R_{1}T(1+B)}.$$
Moreover
$$R(B)\leq R_{1} \left(B+B^{3}\right).$$
\end{theorem}

Let us make the following comments:

\medskip

{\bf a)} The convergence estimate depends on $B=B(T),$ the bound for the $H^{s}$ norm of the exact solution on $[0,T].$ It is well known that the best upper bound for $B(T)$ is a double exponential in time, namely
$$\ln(B(T)) \lesssim \left(1+\ln^{+}\left(\left\|\omega_{0}\right\|_{H^{s}(\mathbbm{T}^{2})}\right)\right)e^{CT}-1,$$
where $\ln^{+}=\ln\mathbbm{1}_{(1+\infty)}.$

\medskip

{\bf b)} Although we used the implicit midpoint integrator, which is known to be in general of order two, the global error scales in $\tau.$ This is due to the freezing effect, {\it ie} the fact that the error in Sobolev space between the solution $f$ of \eqref{eq-frozen} and $\omega$ at time $t$ only scales in $t.$ An interesting perspective would certainly be to reach a global error of order two, for instance by freezing the velocity vector field at the middle of the time step. This will be subject to further investigations.\\
The use of a symplectic integrator is essential in our problem: through area preservation, it ensures that for all $n,$ $\omega_{n}$ has average zero, so that one may solve the Poisson equation with RHS (right-hand side) $\omega_{n},$ and then define rigorously $\omega_{n+1}.$ Although the proof uses as well extensively the special structure of the midpoint integrator, which is the composition of Euler's backward and forward integrator with half time steps, it is therefore possible that our result may be extended to a larger class of symplectic methods.\\
The restriction on the regularity $s\geq 6$ comes from the fact that we shall prove that the local error attributable to the midpoint integrator scales in $\tau^{3},$ as expected. Smaller values of $s$ should be admissible if we are willing to let the local midpoint error to scale in $\tau^{2}$ only, which is the case for the freezing error anyway. In view of \eqref{SD-OP}, one should require the numerical velocity vector field $\mathrm{U}(\omega_{n})$ to be at least Lipschitz, in order to solve the implicit midpoint equation. This should impose at least the restriction $s\geq 4.$

\medskip

{\bf c)} Our proof may be compared with the classical Backward Arror Analysis methods used in geometric numerical integration (see \cite{HLW}): we simply use the fact that the semi-discrete operator $\mathcal{S}_{\tau}(\omega_{0})$ coincides, at time $t=\tau,$ with $\mathcal{S}_{t}(\omega_{0}),$ and it turns out that $\mathcal{S}_{t}(\omega_{0})$ satisfies a transport equation. The local consistency errors are then obtained by means of standard energy estimates for transport equations, with a commutator trick. From that perspective our proof may be related to the paper \cite{CCFM}, where convergence estimates are proved for time-discretizations of the Vlasov-Poisson equation by splitting methods, by means of stability estimates for the associated transport operator.\\
From that perspective, let us say that the result of this paper is not really tied up with the Euler equation, as we really only use the transport structure of the vorticity formulation \eqref{euler2d}. The fact that $\mathrm{U}(\omega)$ belongs to $H^{s+1}$ when $\omega$ belongs to $H^{s}$ is in fact the main feature of the equation that we use in the proof. Therefore our work should apply to a larger class of transport equations, if their velocity vector field satisfies a similar property. Moreover, the divergence-free property of $\mathrm{U}$ does not seem to be mandatory, and having a bounded divergence should be sufficient.

\medskip

{\bf d)} This result concerns time discretizations only. Fully-discrete schemes should moreover involve an interpolation procedure at each step. With periodic boundary conditions it seems natural to use an interpolation by trigonometric polynomials, {\it ie} discrete Fourier series (see \cite{Fourier}). However it is likely that this method involves aliasing errors preventing the stability of the scheme. Another possibility would be to consider the Euler equation on a polygonal domain, and to use Finite Element Method for the interpolation in space. Nevertheless, there exists on such domains solutions of the Euler equation with $H^{2}$ regularity (see section 6 of chapter 3 in \cite{Taylor}), but not better, as far as we know. In view of the discussion on regularity restrictions in {\bf b)}, considering Finite Element Methods should therefore bring technical complications in the proof, where we might need to add a regularization procedure at each step. This will be subject to further investigations.\\

The rest of the paper is organized as follows: In section $3$ we prove general stability estimates for certain transport equations, which will be the main tool of the paper. In section 4 we prove the stability of the semi-discrete operator $\mathcal{S}_{t}.$ In section 5, we analyse the local errors attributable to the freezing of the velocity vector field and to the midpoint discretization, by means described in {\bf c)}, and prove the main result. In the Appendix, we recall, for completion, standard results for the Poisson equation on the torus.

\section{Stability estimates for the exact flows}

\subsection{Notations}

Let $s\geq 2$ and $\omega_{0}\in H^{s}(\mathbbm{T}^{2}).$ The solution $\omega(t,x)$ of the Euler equation with initial data $\omega_{0}$ given by Theorem \ref{existence}, namely
\begin{equation}
\label{exact}
\left\{
\begin{split}
&\partial_{t}\omega - J\nabla \Delta^{-1}\omega \cdot \nabla \omega=0\\
&\omega(t=0,x)=\omega_{0}(x),
\end{split}
\right.
\end{equation}
will be from now on written
\begin{equation}
\label{flow-ex}
\varphi_{E,t}(\omega_{0})(x) = \omega(t,x).
\end{equation}
Let $\psi$ be the solution of the Poisson equation $\Delta \psi = \omega_{0}.$ Proposition \ref{poisson-reg} and the Sobolev embedding \eqref{Sob-inj2} imply that
$$\sup_{0\leq |\alpha| \leq s} \left\| \partial_{x}^{\alpha} \psi \right\|_{L^{\infty}}  \leq C\left\| \psi \right\|_{H^{s+2}(\mathbbm{T}^{2})} \leq C\left\| \omega_{0} \right\|_{H^{s}(\mathbbm{T}^{2})}.$$
The Cauchy-Lipschitz Theorem ensures then that the flow $\Psi_{t}(x)$ associated with the vector field $J\nabla \psi$ is well-defined and exists globally in time, and the function $f(t,x) =\omega_{0}(\Psi_{t}(x))$ solves globally in time the frozen equation
\begin{equation}
\label{glace}
\left\{
\begin{split}
&\partial_{t}f- J\nabla \psi\cdot \nabla f =0 \\
&\Delta \psi =\omega_{0}\\
&f(t=0,x)=\omega_{0}(x)
\end{split}
\right.
\end{equation}
with initial data $\omega_{0}.$ We shall write as previously
\begin{equation}
\label{flow-froz}
\varphi_{F,t}(\omega_{0})(x)=f(t,x),
\end{equation}

\subsection{A stability Lemma for some transport equations}

We first prove, in the Lemma below, estimates for transport operators, that will in particular apply to Euler's equation and to the frozen equation.\\
For a two-dimensional vector field $\rX:\mathbbm{T}^{2}\to \mathbbm{R}^{2},$ let $L(\rX) $ be the operator defined for functions $g$ by
\begin{equation}
\label{transp-OP}
L(\rX)g=\rX\cdot \nabla g,
\end{equation}
and let us define for $\alpha\in\mathbbm{N}^{2}$ the commutator
\begin{equation}
\label{com}
\left[ \partial_{x}^{\alpha}, L(\rX)\right]=\partial_{x}^{\alpha} L(\rX) - L(\rX) \partial_{x}^{\alpha}.
\end{equation}

\begin{lemma}
\label{transport}
For any $s\geq 0,$ there exists a constant $C>0$ such that, for indices $\alpha\in\mathbbm{N}^{2}$ with $|\alpha| \leq s,$ vector fields $\rX$ and functions $g,$ we have
\begin{equation}
\label{transport1}
\left\| \left[ \partial_{x}^{\alpha}, L(\rX)\right] g\right\|_{L^{2}(\mathbbm{T}^{2})} \leq C \left\| \rX \right\|_{H^{s+1}(\mathbbm{T}^{2})} \left[ \left\| g \right\|_{H^{s}(\mathbbm{T}^{2})}  + \left\| g \right\|_{H^{2}(\mathbbm{T}^{2})} \right],
\end{equation}
\begin{equation}
\label{transport1b}
\left\| \left[ \partial_{x}^{\alpha}, L(\rX)\right] g\right\|_{L^{2}(\mathbbm{T}^{2})} \leq C \left\| \rX \right\|_{H^{s+2}(\mathbbm{T}^{2})}\left[ \left\| g \right\|_{H^{s}(\mathbbm{T}^{2})}  + \left\| g \right\|_{H^{1}(\mathbbm{T}^{2})} \right],
\end{equation}
\begin{equation}
\label{transport1c}
\left\| \left[ \partial_{x}^{\alpha}, L(\rX)\right] g\right\|_{L^{2}(\mathbbm{T}^{2})} \leq C \left\| \rX \right\|_{H^{s}(\mathbbm{T}^{2})}\left[ \left\| g \right\|_{H^{s}(\mathbbm{T}^{2})}  + \left\| g \right\|_{H^{3}(\mathbbm{T}^{2})} \right],
\end{equation}
and 
\begin{equation}
\label{transport2}
\left\| \partial_{x}^{\alpha}L(\rX) g \right\|_{L^{2}(\mathbbm{T}^{2})}\leq C \left\| \rX \right\|_{H^{s+1}(\mathbbm{T}^{2})} \left[ \left\| g \right\|_{H^{s+1}(\mathbbm{T}^{2})}  + \left\| g \right\|_{H^{2}(\mathbbm{T}^{2})} \right].
\end{equation}
\end{lemma}

\begin{proof}
We will prove estimates \eqref{transport1}, \eqref{transport1b} and \eqref{transport1c}. The proof of estimate \eqref{transport2} is almost identical to the proof of estimate \eqref{transport1}, and easier, so we will not detail it.\\
All of this is obvious when $|\alpha|=0,$ and for $|\alpha|\geq 1,$ the starting point is to use Leibniz's formula as follows
$$\left[\partial_{x}^{\alpha},L(\rX)\right]g= \partial_{x}^{\alpha}(\rX\cdot \nabla g)- \rX\cdot \partial_{x}^{\alpha} \nabla g=\sum_{\underset{\gamma \neq \alpha}{\beta + \gamma = \alpha}} {\alpha \choose \beta} \partial_{x}^{\beta} \rX\cdot \partial_{x}^{\gamma} \nabla g.$$
Any term of the sum may be estimated using the Sobolev embeddings \eqref{Sob-inj1} or \eqref{Sob-inj2} in several ways, which will produce the three estimates \eqref{transport1}, \eqref{transport1b} and \eqref{transport1c}.

\subsubsection*{Proof of estimate \eqref{transport1}}

When $|\gamma|\geq 1,$ we have by the Sobolev embedding \eqref{Sob-inj2}
\begin{multline*}
\left\| \partial_{x}^{\beta} \rX\cdot \partial_{x}^{\gamma} \nabla g \right\|_{L^{2}(\mathbbm{T}^{2})}\leq \left\| \partial_{x}^{\beta} \rX \right\|_{L^{\infty}(\mathbbm{T}^{2})}\left\| \partial_{x}^{\gamma} \nabla g \right\|_{L^{2}(\mathbbm{T}^{2})} \leq C \left\| \rX\right\|_{H^{|\beta|+2}(\mathbbm{T}^{2})} \left\| g\right\|_{H^{|\gamma|+1}(\mathbbm{T}^{2})} \\
\leq C  \left\| \rX\right\|_{H^{|\alpha|-|\gamma|+2}(\mathbbm{T}^{2})} \left\| g\right\|_{H^{s}(\mathbbm{T}^{2})} \leq C  \left\| \rX\right\|_{H^{s+1}(\mathbbm{T}^{2})} \left\| g\right\|_{H^{s}(\mathbbm{T}^{2})}. 
\end{multline*}
When $|\gamma|=0,$ we have by H\"older's inequality and the Sobolev embedding \eqref{Sob-inj1}
\begin{multline*}
\left\| \partial_{x}^{\beta} \rX\cdot  \nabla g \right\|_{L^{2}(\mathbbm{T}^{2})}\leq \left\| \partial_{x}^{\beta} \rX \right\|_{L^{4}(\mathbbm{T}^{2})}\left\| \nabla g \right\|_{L^{4}(\mathbbm{T}^{2})} \leq C \left\| \rX\right\|_{H^{|\beta|+1}(\mathbbm{T}^{2})} \left\| g\right\|_{H^{2}(\mathbbm{T}^{2})} \\
\leq C  \left\| \rX\right\|_{H^{|\alpha|+1}(\mathbbm{T}^{2})} \left\| g\right\|_{H^{2}(\mathbbm{T}^{2})} \leq C  \left\| \rX\right\|_{H^{s+1}(\mathbbm{T}^{2})} \left\| g\right\|_{H^{2}(\mathbbm{T}^{2})}.
\end{multline*}
This proves estimate \eqref{transport1}.

\subsubsection*{Proof of estimate \eqref{transport1b}}

It suffices to use Sobolev's embedding \eqref{Sob-inj2} for any $\gamma \neq \alpha,$ as follows
\begin{multline*}
\left\| \partial_{x}^{\beta} \rX\cdot \partial_{x}^{\gamma} \nabla g \right\|_{L^{2}(\mathbbm{T}^{2})}\leq \left\| \partial_{x}^{\beta} \rX \right\|_{L^{\infty}(\mathbbm{T}^{2})}\left\| \partial_{x}^{\gamma} \nabla g \right\|_{L^{2}(\mathbbm{T}^{2})} \leq C \left\| \rX\right\|_{H^{|\beta|+2}(\mathbbm{T}^{2})} \left\| g\right\|_{H^{|\gamma|+1}(\mathbbm{T}^{2})} \\
\leq C  \left\| \rX\right\|_{H^{|\alpha|-|\gamma|+2}(\mathbbm{T}^{2})} \left\| g\right\|_{H^{|\gamma|+1}(\mathbbm{T}^{2})} \leq C  \left\| \rX\right\|_{H^{s+2-|\gamma|}(\mathbbm{T}^{2})} \left\| g\right\|_{H^{|\gamma| + 1}(\mathbbm{T}^{2})}. 
\end{multline*}
For any $|\gamma| >0,$ the RHS is 
$$C  \left\| \rX\right\|_{H^{s+2-|\gamma|}(\mathbbm{T}^{2})} \left\| g\right\|_{H^{|\gamma| + 1}(\mathbbm{T}^{2})}\leq C  \left\| \rX\right\|_{H^{s+2}(\mathbbm{T}^{2})} \left\| g\right\|_{H^{s}(\mathbbm{T}^{2})},$$
and for $|\gamma|=0,$ it is 
$$C  \left\| \rX\right\|_{H^{s+2}(\mathbbm{T}^{2})} \left\| g\right\|_{H^{1}(\mathbbm{T}^{2})}.$$
This proves \eqref{transport1b}.

\subsubsection*{Proof of estimate \eqref{transport1c}}

For any $|\gamma| \geq 2,$ we may use the Sobolev embedding \eqref{Sob-inj2} as previously
\begin{multline*}
\left\| \partial_{x}^{\beta} \rX\cdot \partial_{x}^{\gamma} \nabla g \right\|_{L^{2}(\mathbbm{T}^{2})}\leq \left\| \partial_{x}^{\beta} \rX \right\|_{L^{\infty}(\mathbbm{T}^{2})}\left\| \partial_{x}^{\gamma} \nabla g \right\|_{L^{2}(\mathbbm{T}^{2})} \leq C \left\| \rX\right\|_{H^{|\beta|+2}(\mathbbm{T}^{2})} \left\| g\right\|_{H^{|\gamma|+1}(\mathbbm{T}^{2})} \\
\leq C  \left\| \rX\right\|_{H^{|\alpha|-|\gamma|+2}(\mathbbm{T}^{2})} \left\| g\right\|_{H^{|\gamma|+1}(\mathbbm{T}^{2})} \leq C  \left\| \rX\right\|_{H^{s}(\mathbbm{T}^{2})} \left\| g\right\|_{H^{s}(\mathbbm{T}^{2})}. 
\end{multline*}
When $|\gamma|=1,$ we may use the Sobolev embedding \eqref{Sob-inj1} and H\"older's inequality
\begin{multline*}
\left\| \partial_{x}^{\beta} \rX\cdot  \partial_{x}^{\gamma} \nabla  g \right\|_{L^{2}(\mathbbm{T}^{2})}\leq \left\| \partial_{x}^{\beta} \rX \right\|_{L^{4}(\mathbbm{T}^{2})}\left\| \partial_{x}^{\gamma} \nabla g \right\|_{L^{4}(\mathbbm{T}^{2})} \leq C \left\| \rX\right\|_{H^{|\beta|+1}(\mathbbm{T}^{2})} \left\| g\right\|_{H^{3}(\mathbbm{T}^{2})} \\
\leq C  \left\| \rX\right\|_{H^{|\alpha| - |\gamma|+1}(\mathbbm{T}^{2})} \left\| g\right\|_{H^{3}(\mathbbm{T}^{2})} \leq C  \left\| \rX\right\|_{H^{s}(\mathbbm{T}^{2})} \left\| g\right\|_{H^{3}(\mathbbm{T}^{2})}.
\end{multline*}
Finally, when $|\gamma|=0,$ Sobolev's embedding \eqref{Sob-inj2} gives us the estimate
\begin{multline*}
\left\| \partial_{x}^{\beta} \rX\cdot  \nabla g \right\|_{L^{2}(\mathbbm{T}^{2})}\leq \left\| \partial_{x}^{\beta} \rX \right\|_{L^{2}(\mathbbm{T}^{2})}\left\| \nabla g \right\|_{L^{\infty}(\mathbbm{T}^{2})} \leq C \left\| \rX\right\|_{H^{|\beta|}(\mathbbm{T}^{2})} \left\| g\right\|_{H^{3}(\mathbbm{T}^{2})} \\
\leq C  \left\| \rX\right\|_{H^{|\alpha|}(\mathbbm{T}^{2})} \left\| g\right\|_{H^{3}(\mathbbm{T}^{2})} \leq C  \left\| \rX\right\|_{H^{s}(\mathbbm{T}^{2})} \left\| g\right\|_{H^{3}(\mathbbm{T}^{2})}.
\end{multline*}
Collecting the three previous estimates yields \eqref{transport1c}.
\end{proof}

The previous Lemma allows us to obtain the following stability result, which will be the main technical tool of the paper.

\begin{lemma}
\label{stability}
For any $s\geq 0,$ there exists a constant $C>0$ such that, for vector fields $\rX$ and functions $h,$ if $g$ solves the equation
\begin{equation}
\notag
\partial_{t}g - \rX\cdot \nabla g =h.
\end{equation}
then for all $t\in \mathbbm{R},$ $g$ enjoys the estimates
\begin{equation}
\label{EE1}
\begin{split}
\frac{\dd}{\dd t}\left\| g(t)\right\|_{H^{s}(\mathbbm{T}^{2})}^{2}\leq & C \left[ \left\| \rX(t)\right\|_{H^{s+1}(\mathbbm{T}^{2})} \left( \left\|g(t)\right\|_{H^{s}(\mathbbm{T}^{2})} + \left\|g(t)\right\|_{H^{2}(\mathbbm{T}^{2})}  \right) \right]\left\|g(t)\right\|_{H^{s}(\mathbbm{T}^{2})}\\
& + \left\| \nabla \cdot \rX(t)\right\|_{L^{\infty}(\mathbbm{T}^{2})}\left\|g(t)\right\|_{H^{s}(\mathbbm{T}^{2})}^{2}  + 2  \left\| h(t)\right\|_{H^{s}(\mathbbm{T}^{2})}  \left\| g(t)\right\|_{H^{s}(\mathbbm{T}^{2})} ,
\end{split}
\end{equation}
\begin{equation}
\label{EE2}
\begin{split}
\frac{\dd}{\dd t}\left\| g(t)\right\|_{H^{s}(\mathbbm{T}^{2})}^{2}\leq & C \left[ \left\| \rX(t)\right\|_{H^{s+2}(\mathbbm{T}^{2})} \left( \left\|g(t)\right\|_{H^{s}(\mathbbm{T}^{2})} + \left\|g(t)\right\|_{H^{1}(\mathbbm{T}^{2})}  \right) \right]\left\|g(t)\right\|_{H^{s}(\mathbbm{T}^{2})}\\
& + \left\| \nabla \cdot \rX(t)\right\|_{L^{\infty}(\mathbbm{T}^{2})}\left\|g(t)\right\|_{H^{s}(\mathbbm{T}^{2})}^{2}  + 2  \left\| h(t)\right\|_{H^{s}(\mathbbm{T}^{2})}  \left\| g(t)\right\|_{H^{s}(\mathbbm{T}^{2})} ,
\end{split}
\end{equation}
and
\begin{equation}
\label{EE3}
\begin{split}
\frac{\dd}{\dd t}\left\| g(t)\right\|_{H^{s}(\mathbbm{T}^{2})}^{2}\leq & C \left[ \left\| \rX(t)\right\|_{H^{s}(\mathbbm{T}^{2})} \left( \left\|g(t)\right\|_{H^{s}(\mathbbm{T}^{2})} + \left\|g(t)\right\|_{H^{3}(\mathbbm{T}^{2})}  \right) \right]\left\|g(t)\right\|_{H^{s}(\mathbbm{T}^{2})}\\
& + \left\| \nabla \cdot \rX(t)\right\|_{L^{\infty}(\mathbbm{T}^{2})}\left\|g(t)\right\|_{H^{s}(\mathbbm{T}^{2})}^{2}  + 2  \left\| h(t)\right\|_{H^{s}(\mathbbm{T}^{2})}  \left\| g(t)\right\|_{H^{s}(\mathbbm{T}^{2})}.
\end{split}
\end{equation}
\end{lemma}

\begin{proof}
We prove the Lemma by an energy estimate. With the notations introduced in \eqref{transp-OP} and \eqref{com}, we have for all $|\alpha|\leq s,$
\begin{equation}
\notag
\begin{split}
\frac{1}{2}\frac{\dd}{\dd t} \left\| \partial_{x}^{\alpha} g(t)\right\|_{L^{2}(\mathbbm{T}^{2})}^{2} &=  \langle \partial_{x}^{\alpha} L(\rX)g ,\partial_{x}^{\alpha} g\rangle_{L^{2}(\mathbbm{T}^{2})} + \langle \partial_{x}^{\alpha} h ,\partial_{x}^{\alpha} g\rangle_{L^{2}(\mathbbm{T}^{2})} \\
& = \langle \left[\partial_{x}^{\alpha}, L(\rX)\right]g ,\partial_{x}^{\alpha} g\rangle_{L^{2}(\mathbbm{T}^{2})} + \langle  L(\rX) \partial_{x}^{\alpha}g ,\partial_{x}^{\alpha} g\rangle_{L^{2}(\mathbbm{T}^{2})} + \langle \partial_{x}^{\alpha} h ,\partial_{x}^{\alpha} g\rangle_{L^{2}(\mathbbm{T}^{2})}.
\end{split}
\end{equation}
The last term is obviously controlled as follows
$$\left| \langle \partial_{x}^{\alpha} h ,\partial_{x}^{\alpha} g\rangle_{L^{2}(\mathbbm{T}^{2})} \right| \leq \left\| h(t)\right\|_{H^{s}(\mathbbm{T}^{2})}  \left\| g(t)\right\|_{H^{s}(\mathbbm{T}^{2})}.$$
Also,
$$\langle  L(\rX) \partial_{x}^{\alpha}g ,\partial_{x}^{\alpha} g\rangle_{L^{2}(\mathbbm{T}^{2})}= \int_{\mathbbm{T}^{2}} \rX(t,x)\cdot \nabla \partial_{x}^{\alpha}g(t,x) \overline{\partial_{x}^{\alpha}g(t,x)} \dd x = -\frac{1}{2} \int_{\mathbbm{T}^{2}} \nabla \cdot \rX(t,x) |\partial_{x}^{\alpha}g(t,x)|^{2} \dd x,$$
such that
$$\left| \langle  L(\rX) \partial_{x}^{\alpha}g ,\partial_{x}^{\alpha} g\rangle_{L^{2}(\mathbbm{T}^{2})} \right| \leq \frac{1}{2} \left\| \nabla \cdot \rX(t)\right\|_{L^{\infty}(\mathbbm{T}^{2})}\left\|g(t)\right\|_{H^{s}(\mathbbm{T}^{2})}^{2}.$$
Finally,
$$\left| \langle \left[\partial_{x}^{\alpha}, L(\rX)\right]g ,\partial_{x}^{\alpha} g\rangle_{L^{2}(\mathbbm{T}^{2})} \right| \leq \left\| \left[\partial_{x}^{\alpha}, L(\rX)\right]g \right\|_{L^{2}(\mathbbm{T}^{2})} \left\|g(t)\right\|_{H^{s}(\mathbbm{T}^{2})},$$
such that estimates \eqref{transport1}, \eqref{transport1b} and \eqref{transport1c} from Lemma \ref{transport} yields respectively the estimates \eqref{EE1}, \eqref{EE2} and \eqref{EE3}.
\end{proof}

\subsection{Stability estimates}

We will prove in the Proposition below the stability in $H^{s}$ of the flows $\varphi_{E,t}$ and $\varphi_{F,t},$ defined respectively by \eqref{exact} {\normalfont{\&}} \eqref{flow-ex}, and \eqref{glace} {\normalfont{\&}} \eqref{flow-froz}.\\
Throughout this subsection we shall use the following property: for any $s\geq 0$ and function $g\in H^{s}(\mathbbm{T}^{2}),$ the vector field $J\nabla \Delta^{-1}g$ enjoys the estimate
\begin{equation}
\label{poiss}
\left\| J\nabla \Delta^{-1} g\right\|_{H^{s+1}(\mathbbm{T}^{2})} \leq C \left\| g \right\|_{H^{s}(\mathbbm{T}^{2})}.
\end{equation}
This is an easy consequence of Proposition \ref{poisson-reg}.

\begin{proposition}
\label{stab}
Let $s\geq 2,$ $\omega_{0}\in H^{s}(\mathbbm{T}^{2})$ with average $0,$ and $B>0$ such that 
$\left\|\omega_{0}\right\|_{H^{s}(\mathbbm{T}^{2})}\leq B.$ There exists two positive constants $L_{0}$ and $L_{1},$ both independent of $\omega_{0},$ such that, if 
$$T_{0}<\frac{1}{L_{0}B},$$
then for all $t\in[0,T_{0}],$
\begin{equation}
\label{stab-ex}
\left\|\varphi_{E,t}(\omega_{0})\right\|_{H^{s}(\mathbbm{T}^{2})}\leq \min\left(2,e^{BL_{1}t}\right) \left\| \omega_{0} \right\|_{H^{s}(\mathbbm{T}^{2})},
\end{equation}
and
\begin{equation}
\label{stab-froz}
\left\|\varphi_{F,t}(\omega_{0})\right\|_{H^{s}(\mathbbm{T}^{2})}\leq \min\left(2,e^{BL_{1}t}\right) \left\| \omega_{0} \right\|_{H^{s}(\mathbbm{T}^{2})}.
\end{equation}
\end{proposition}

\begin{proof}
We begin with the stability estimate \eqref{stab-ex}. $\varphi_{E,t}(\omega_{0})$ satisfies the equation
$$\partial_{t} \varphi_{E,t}(\omega_{0}) -J\nabla \Delta^{-1}\varphi_{E,t}(\omega_{0})\cdot \nabla\varphi_{E,t}(\omega_{0})=0.$$
Applying Lemma \ref{stability} with the divergence-free vector field $\rX(t)=J\nabla \Delta^{-1}\varphi_{E,t}(\omega_{0}),$ estimate \eqref{EE1} (with $s\geq 2$) implies that for all $t\geq 0,$
$$\left\| \varphi_{E,t}(\omega_{0}) \right\|_{H^{s}(\mathbbm{T}^{2})}^{2} \leq \left\| \omega_{0}\right\|_{H^{s}(\mathbbm{T}^{2})}^{2} + C\int_{0}^{t} \left\| \varphi_{E,\sigma}(\omega_{0})\right\|_{H^{s}(\mathbbm{T}^{2})} \left\| \varphi_{E,\sigma}(\omega_{0})\right\|_{H^{s}(\mathbbm{T}^{2})}^{2} \dd \sigma.$$
Here we have also used the inequality \eqref{poiss} with the function $g=\varphi_{E,t}(\omega_{0}).$\\
If $T_{0}>0$ is such that the following estimate holds,
$$\sup_{t\in[0,T_{0}]} \left\| \varphi_{E,t}(\omega_{0}) \right\|_{H^{s}(\mathbbm{T}^{2})} \leq 2\left\|\omega_{0}\right\|_{H^{s}(\mathbbm{T}^{2})},$$
the previous inequality implies then that
$$(1-2BCT_{0})\sup_{t\in [0,T_{0}]} \left\| \varphi_{E,t}(\omega_{0}) \right\|_{H^{s}(\mathbbm{T}^{2})}^{2}\leq \left\| \omega_{0} \right\|_{H^{s}(\mathbbm{T}^{2})}^{2}.$$
Hence if $T_{0}$ is chosen such that $2BCT_{0}<3/4,$ we obtain the estimate
$$\sup_{t\in[0,T_{0}]}\left\| \varphi_{E,t}(\omega_{0}) \right\|_{H^{s}(\mathbbm{T}^{2})}< 2\left\| \omega_{0} \right\|_{H^{s}(\mathbbm{T}^{2})}.$$
By a bootstrap argument this implies that a time $T_{0}>0$ can be chosen such that, if $2BCT_{0}<3/4,$ $\varphi_{E,t}(\omega_{0})$ enjoys the estimate
$$\sup_{t\in[0,T_{0}]} \left\| \varphi_{E,t}(\omega_{0}) \right\|_{H^{s}(\mathbbm{T}^{2})} \leq 2\left\| \omega_{0} \right\|_{H^{s}(\mathbbm{T}^{2})}.$$
Using Gronwall's Lemma we infer easily that
$$\left\| \varphi_{E,t}(\omega_{0}) \right\|_{H^{s}(\mathbbm{T}^{2})}\leq e^{BL_{1}t} \left\| \omega_{0} \right\|_{H^{s}(\mathbbm{T}^{2})},$$
which proves estimate \eqref{stab-ex}, with $L_{0}= 8C/3,$ and $L_{1}=2C.$\\
To prove estimate \eqref{stab-froz}, we use the fact that $\varphi_{F,t}(\omega_{0})$ satisfies the equation
$$\partial_{t}\varphi_{F,t}(\omega_{0}) - J \nabla  \Delta^{-1}\omega_{0}\cdot \nabla\varphi_{F,t}(\omega_{0})=0,$$
with initial data $\omega_{0}.$ Applying once more Lemma \ref{stability} with the divergence-free vector field $\rX(t)=J\nabla \Delta^{-1}\omega_{0},$
estimate \eqref{EE1} yields as previously
$$\left\| \varphi_{F,t}(\omega_{0})\right\|_{H^{s}(\mathbbm{T}^{2})}^{2} \leq \left\| \omega_{0}\right\|_{H^{s}(\mathbbm{T}^{2})}^{2} + C\int_{0}^{t} \left\| \omega_{0}\right\|_{H^{s}(\mathbbm{T}^{2})} \left\| \varphi_{F,\sigma}(\omega_{0})\right\|_{H^{s}(\mathbbm{T}^{2})}^{2} \dd \sigma,$$
where we have also applied the inequality \eqref{poiss} with the function $g=\omega_{0}.$\\
Hence if $2BCT_{0}<3/4,$ it implies that for all $t\in[0,T_{0}],$
$$\left\| \varphi_{F,t}(\omega_{0})\right\|_{H^{s}(\mathbbm{T}^{2})} \leq 2 \left\| \omega_{0}\right\|_{H^{s}(\mathbbm{T}^{2})}.$$
On the other hand, Gronwall's Lemma implies that
$$\left\| \varphi_{F,t}(\omega_{0})\right\|_{H^{s}(\mathbbm{T}^{2})}\leq e^{BL_{1}t}\left\| \omega_{0}\right\|_{H^{s}(\mathbbm{T}^{2})}.$$
Therefore,
$$\left\|\varphi_{F,t}(\omega_{0})\right\|_{H^{s}(\mathbbm{T}^{2})}\leq \min\left(2,e^{BL_{1}t}\right) \left\| \omega_{0} \right\|_{H^{s}(\mathbbm{T}^{2})},$$
and estimate \eqref{stab-froz} is proven.
\end{proof}

\section{Numerical stability}

\subsection{Properties of the midpoint flow}

\begin{proposition}
\label{Mid-topo}
Let  $s\geq 3,$ $\omega_{0} \in H^{s}(\mathbbm{T}^{2})$ with average $0.$ Let $\psi$ be the solution of the Poisson equation $\Delta \psi =\omega_{0},$ and let $\tau \in ]0,1[.$ There exists a positive constant $R_{0},$ independent of $\omega_{0},$ such that, if $\tau \left\| \omega_{0}\right\|_{H^{2}(\mathbbm{T}^{2})}R_{0}<1,$ the following properties hold:

\medskip

{\bf i)} For all $t\in [0,\tau]$ and $x\in\mathbbm{T}^{2},$ there exists an unique solution $\Phi_{t}(x)$ to the implicit equation
\begin{equation}
\label{mid}
\Phi_{t}(x) =x + tJ\nabla\psi\left(\frac{x+\Phi_{t}(x)}{2}\right).
\end{equation}

\medskip

{\bf ii)} The function $\Phi_{t}(x):[0,\tau]\times \mathbbm{T}^{2} \to \mathbbm{T}^{2}$ is $C^{s-1},$ and for all $t\in [0,\tau],$ $x\mapsto \Phi_{t}(x)$ is a symplectic global diffeomorphism on $\mathbbm{T}^{2}.$
Moreover, $\Phi_{t}^{-1}=\Phi_{-t}.$

\medskip

{\bf iii)} For all $t\in[0,\tau],$ the mappings
$$\quad \Ec_{t}(x)=x+\frac{t}{2}J\nabla \psi(x) \quad \mbox{and} \quad \Ecs_{t}(x) = \Ec_{-t}^{-1}(x)$$
are as well global diffeomorphisms on $\mathbbm{T}^{2},$ and
$$\Phi_{t}=\Ec_{t}\circ \Ecs_{t}.$$

\medskip

{\bf iv)} Let $V(t,x)$ be the vector field defined by
$$V(t,x)=\partial_{t} \Phi_{-t} \circ \Phi_{t}(x).$$
If $s\geq 4,$ there exists a $C^{s-4}$ vector field $(t,x)\mapsto \mathcal{R}(t,x)$ such that
$$V(t,x)=-J\nabla \psi(x) + t^{2} \mathcal{R}(t,x).$$
\end{proposition}

\begin{proof}
\hspace{0mm}
\subsubsection*{Proof of assertion {\bf i)}}

Let us first note that, using Proposition \ref{poisson-reg} and the Sobolev embedding \eqref{Sob-inj2}, we have for any $\sigma \geq 0,$
\begin{equation}
\label{reg}
\sup_{0\leq |\alpha|\leq \sigma} \left\| \partial_{x}^{\alpha}\psi \right\|_{L^{\infty}(\mathbbm{T}^{2})}\leq C\left\| \psi\right\|_{H^{\sigma+2}(\mathbbm{T}^{2})} \leq C\left\| \omega_{0}\right\|_{H^{\sigma}(\mathbbm{T}^{2})}.
\end{equation}
The existence (and uniqueness) of $\Phi_{t}(x)$ follows then: for $t\in [0,\tau],$ and $x\in \mathbbm{T}^{2},$ we define a function $F_{t,x}:\mathbbm{T}^{2}\to \mathbbm{T}^{2}$ by
$$F_{t,x}(y)=x+t J\nabla\psi \left(\frac{x+y}{2}\right).$$
For any $y,\tilde{y}\in\mathbbm{T}^{2},$ we have by the mean-value Theorem
$$\left|F_{t,x}(y)-F_{t,x}(\tilde{y})\right|=t\left| J\nabla\psi \left(\frac{x+y}{2}\right) - J\nabla\psi \left(\frac{x+\tilde{y}}{2}\right)\right|\leq \tau C \sup_{|\alpha|=2}\left\| \partial_{x}^{\alpha} \psi\right\|_{L^{\infty}(\mathbbm{T}^{2})} |y-\tilde{y}| < |y-\tilde{y}|,$$
provided that, using \eqref{reg}, $\tau R_{0} \left\| \omega_{0} \right\|_{H^{2}(\mathbbm{T}^{2})}<1,$ for some appropriate constant $R_{0}=R_{0}(C).$ In that case $F_{t,x}$ is a contraction mapping on $\mathbbm{T}^{2},$ such that Banach's fixed point Theorem gives us an unique solution $\Phi_{t}(x)$ to the equation
$$F_{t,x}\left(\Phi_{t}(x)\right)=\Phi_{t}(x).$$
This proves the assertion {\bf i)}.

\subsubsection*{Proof of assertion {\bf ii)}}

Let us now consider, for $\eps\in ]0,1[,$ the function $G:]-\eps,\tau+\eps[\times \mathbbm{T}^{2} \times \mathbbm{T}^{2}\to \mathbbm{T}^{2}$ defined by
\begin{equation}
\label{G}
G(t,y,x)=y-x-tJ\nabla\psi\left(\frac{x+y}{2}\right).
\end{equation}
Then by \eqref{reg}, $\nabla \psi$ belongs to $H^{s+1},$ which is continuously embedded in $C^{s-1},$ such that $G$ has class $C^{s-1}$ and, in addition, for all $(t,x)\in[0,\tau] \times \mathbbm{T}^{2},$
$$G(t,\Phi_{t}(x),x)=0.$$
Moreover,
$$\rD_{y}G(t,\Phi_{t}(x),x)=A_{t}\left(JY^{t}(x)\right),$$
with
\begin{equation}
\label{def-Y}
Y^{t}(x)=\nabla^{2}\psi\left(\frac{\Phi_{t}(x)+x}{2}\right),
\end{equation}
and where, if $Y$ is a $2\times 2$ square matrix,
\begin{equation}
\label{def-A}
A_{t}(Y)=I_{2}-\frac{t}{2}Y,
\end{equation}
$I_{2}$ being the identity matrix of $M_{2}(\mathbbm{R}).$\\
Thanks to \eqref{reg}, 
$$\sup_{|\alpha|=2} \left\| \partial_{x}^{\alpha}\psi\right\|_{L^{\infty}(\mathbbm{T}^{2})} \leq  C\left\| \omega_{0} \right\|_{H^{2}(\mathbbm{T}^{2})},$$
and thus it is well-known that, if $(\tau+\eps) C\left\| \omega_{0} \right\|_{H^{2}(\mathbbm{T}^{2})} <2,$ $A_{t}\left(JY^{t}(x)\right)$ is invertible for all $(t,x)\in[0,\tau]\times \mathbbm{T}^{2}$ and
\begin{equation}
\label{inverse-A}
A_{t}\left(JY^{t}(x)\right)^{-1}=\sum_{n =0}^{+\infty} \frac{t^{n}}{2^{n}} \left(JY^{t}(x)\right)^{n}.
\end{equation}
We may assume that this is true, given the assumption on $\tau,$ and choosing $\eps$ small enough.\\
As $s-1>1,$ we can apply the implicit function Theorem, which shows then that the function $(t,x)\mapsto \Phi_{t}(x)$ has $C^{s-1}$ regularity on $]-\eps,\tau+\eps[\times \mathbbm{T}^{2}.$\\
In addition we are allowed to differentiate the equation
$$G(t,\Phi_{t}(x),x)=0$$
with respect to $x,$ and it implies that
\begin{equation}
\label{mid-dz}
A_{t}(JY^{t}(x))\rD_{x}\Phi_{t}(x)=A_{-t}(JY^{t}(x)),
\end{equation}
Since $A_{t}(JY^{t}(x))$ and $A_{-t}(JY^{t}(x))$ are invertible, then so is $\rD_{x}\Phi_{t}(x).$ By the local inverse Theorem and the open mapping Theorem, $\Phi_{t}(\cdot)$ is therefore a local diffeomorphism on $\mathbbm{T}^{2},$ and an open mapping. In particular $\Phi_{t}\left(\mathbbm{T}^{2}\right)$ is open, and, by continuity, also compact, thus closed. By connectedness, we conclude that $\Phi_{t}(\cdot)$ is onto. It is also one-to-one, since if $\Phi_{t}(x)=\Phi_{t}(y),$ then by the mean-value Theorem
$$|x-y|=\left|t J\nabla\psi\left(\frac{x + \Phi_{t}(x)}{2}\right) - t J\nabla\psi\left(\frac{y + \Phi_{t}(y) }{2}\right) \right|\leq \frac{tC\left\| \omega_{0} \right\|_{H^{2}(\mathbbm{T}^{2})}}{2} |x-y|,$$
which implies that $x=y$ if $t\left\| \omega_{0} \right\|_{H^{2}(\mathbbm{T}^{2})}C<2.$ Thus $\Phi_{t}(\cdot)$ is a global diffeomorphism on $\mathbbm{T}^{2}.$\\
It remains to show that it is symplectic, {\it ie} that
$$\rD_{x}\Phi_{t}(x)^{\top} J \rD_{x}\Phi_{t}(x)=J.$$
Using \eqref{mid-dz}, the identity $J^{\top}=J^{-1}=-J$ and the symmetry of the matrix $Y^{t}(x),$ we have
\begin{equation}
\notag
\begin{split}
&\rD_{x}\Phi_{t}(x)^{\top} J \rD_{x}\Phi_{t}(x)=J\\
\Leftrightarrow \quad & A_{-t}(JY^{t}(x))^{\top} \left(A_{t}(JY^{t}(x))^{\top}\right)^{-1} J \left(A_{t}(JY^{t}(x))\right)^{-1} A_{-t}(JY^{t}(x)) = J \\
\Leftrightarrow \quad & \left(A_{t}(JY^{t}(x))^{\top}\right)^{-1} J \left(A_{t}(JY^{t}(x))\right)^{-1} = \left(A_{-t}(JY^{t}(x))^{\top}\right)^{-1} J \left(A_{-t}(JY^{t}(x))\right)^{-1} \\
\Leftrightarrow \quad & A_{t}(JY^{t}(x)) J A_{t}(JY^{t}(x))^{\top} = A_{-t}(JY^{t}(x)) J A_{-t}(JY^{t}(x))^{\top} \\
\Leftrightarrow \quad &  A_{t}(JY^{t}(x)) J A_{-t}(Y^{t}(x)J) = A_{-t}(JY^{t}(x)) J A_{t}(Y^{t}(x)J) \\ 
\Leftrightarrow \quad & J=J,
\end{split}
\end{equation}
the last line being easily obtained by expanding each sides of the penultimate equality.\\
Finally, as we have for all $x\in\mathbbm{T}^{2},$
$$\Phi_{t}(x)=x+tJ\nabla\psi\left(\frac{x+\Phi_{t}(x)}{2}\right),$$
one infers that for all $x\in\mathbbm{T}^{2},$
$$x=\Phi_{t}^{-1}(x)+tJ\nabla\psi\left(\frac{x+\Phi_{t}^{-1}(x)}{2}\right),$$
which shows that $\Phi_{t}^{-1}=\Phi_{-t}.$

\subsubsection*{Proof of assertion {\bf iii)}}

The mappings $\Ec_{t}$ and $\Ecs_{t}$ are defined by
$$\Ec_{t}(x)=x+\frac{t}{2}J\nabla \psi(x) \quad \mbox{and} \quad \Ecs_{t}(x)=x+\frac{t}{2}J\nabla\psi(\Ecs_{t}(x)).$$
Thus, using the above notation \eqref{def-A},
$$\rD_{x}\Ec_{t}(x)=A_{-t}(J\nabla^{2}\psi(x))\quad \mbox{and} \quad A_{t}(J\nabla ^{2} \psi (\Ecs_{t}(x)))\rD_{x}\Ecs_{t}(x)= I_{2},$$
such that we may repeat the previous arguments (local inverse Theorem, open mapping Theorem) to conclude that $\Ec_{t}$ and $\Ecs_{t}$ are global diffeomorphisms on $\mathbbm{T}^{2}.$\\
Moreover, for all $t\in [0,t],$ $y=\Ecs_{t}(x)$ is by definition the unique solution of the equation
$$y=x+\frac{t}{2}J\nabla\psi(y),$$
which is also solved by $y= \frac{x+\Phi_{t}(x)}{2}.$ Hence $\Ecs_{t}(x)= \frac{x+\Phi_{t}(x)}{2},$
and
$$\Phi_{t}=\Ec_{t}\circ \Ecs_{t}.$$

\subsubsection*{Proof of assertion {\bf iv)}}

In that part of the proof we shall need the following derivatives of the function $G$ defined by \eqref{G}:
\begin{equation}
\label{DG}
\left\{
\begin{split}
&\partial_{t}G(t,y,x) = -J\nabla\psi \left(\frac{x+y}{2}\right) \\
&\rD_{y} G(t,y,x)= I_{2} - \frac{t}{2} J\nabla^{2}\psi \left(\frac{x+y}{2}\right) \\
&\rD_{x} G(t,y,x)= -I_{2} - \frac{t}{2} J\nabla^{2}\psi \left(\frac{x+y}{2}\right) \\
&\rD_{x}\partial_{t} G(t,y,x)= \rD_{y}\partial_{t} G(t,y,x)=-\frac{1}{2} J\nabla^{2}\psi \left(\frac{x+y}{2}\right) \\
&\rD_{x}\rD_{y}G(t,y,x) = -\frac{t}{4} J\nabla^{3}\psi \left(\frac{x+y}{2}\right) \\
& \partial_{t}^{2}G(t,y,x) = (0,0)^{\top}.
\end{split}
\right. 
\end{equation}
We will write the second order Taylor-expansion in time of $ \partial_{t}(\Phi_{-t}(x)) \circ \Phi_{t}(x),$ and for that we will need the expressions of
$$\Phi_{t}(x),\quad \partial_{t}\Phi_{t}(x), \quad\partial_{t}(\Phi_{-t}(x)) \circ \Phi_{t}(x),\quad \frac{\dd}{\dd t} [\partial_{t}(\Phi_{-t}(x)) \circ \Phi_{t}(x)]$$
at time $t=0.$\\
First of all, using \eqref{G} and \eqref{DG} and evaluating the identities
$$G(t,\Phi_{t}(x),x)=0 \quad \mbox{and} \quad \partial_{t}G(t,\Phi_{t}(x),x)+ \rD_{y}G(t,\Phi_{t}(x),x)\partial_{t}\Phi_{t}(x)=0$$
at $t=0$ gives us
\begin{equation}
\label{step1}
\Phi^{0}(x)=x\quad \mbox{and}\quad \partial_{t}\Phi_{t}(x)_{|t=0}=J\nabla\psi(x).
\end{equation}
In addition, we know that
\begin{equation}
\label{-G}
G(-t,\Phi_{-t}(x),x)=0.
\end{equation}
Differentiating \eqref{-G} with respect to the time, we obtain 
$$-\partial_{t}G(-t,\Phi_{-t}(x),x) +\rD_{y}G(-t,\Phi_{-t}(x),x) \partial_{t}(\Phi_{-t}(x))=0.$$
It holds for all $x\in\mathbbm{T}^{2},$ and thus, pulling-back by the map $x\mapsto \Phi_{t}(x),$ we infer that
\begin{equation}
\label{eq-DG2}
-\partial_{t}G(-t,x,\Phi_{t}(x)) + \rD_{y}G(-t,x,\Phi_{t}(x)) \left[\partial_{t}(\Phi_{-t}(x))\circ \Phi_{t}(x)\right]=0.
\end{equation}
Evaluating \eqref{eq-DG2} at $t=0$ and using \eqref{DG}, we obtain
\begin{equation}
\label{step2}
\left[\partial_{t}(\Phi_{-t}(x)) \circ \Phi_{t}(x)\right]_{|t=0}=-J\nabla \psi(x).
\end{equation}
Differentiating \eqref{eq-DG2} with respect to $t$ we have, 
\begin{multline*}
-\rD_{x}\partial_{t} G(-t,x,\Phi_{t}(x)) \partial_{t}\Phi_{t}(x)-\partial_{t}\rD_{y} G(-t,x,\Phi_{t}(x)) \left[\partial_{t}(\Phi_{-t}(x)) \circ \Phi_{t}(x)\right] \\
\hspace{36mm}+ \rD_{x}\rD_{y} G(-t,x,\Phi_{t}(x)) \partial_{t}\Phi_{t}(x)\cdot\left[\partial_{t}(\Phi_{-t}(x)) \circ \Phi_{t}(x)\right]\\
+\rD_{y}G(t,\Phi_{t}(x),x) \frac{\dd}{\dd t} \left[\partial_{t}(\Phi_{-t}(x)) \circ \Phi_{t}(x)\right]=0.
\end{multline*}
Evaluating this expression at $t=0$ with the help of \eqref{DG}, \eqref{step1} and \eqref{step2} gives us
$$\frac{\dd}{\dd t} \left[\partial_{t}(\Phi_{-t}(x)) \circ \Phi_{t}(x)\right]_{|t=0}=0.$$
Using this and \eqref{step2}, we conclude by a Taylor expansion that for all $t\in [0,\tau],$
$$\partial_{t}(\Phi_{-t}(x)) \circ \Phi_{t}(x)=-J\nabla \psi(x) + \frac{t^{2}}{2}\mathcal{R}(t,x).$$
Moreover the Taylor remainder has the regularity of 
$$\frac{\dd^{2}}{\dd t^{2}} \left[\partial_{t}(\Phi_{-t}(x)) \circ \Phi_{t}(x)\right],$$
which is $C^{s-4},$ as $(t,x)\mapsto \partial_{t}(\Phi_{-t}(x))$ is $C^{s-2}$ and $(t,x)\mapsto \Phi_{t}(x)$ is $C^{s-1}.$
\end{proof}

\begin{remark}
\label{divzero}
In particular, if a function $g$ has average $0,$ then the function $g\circ \Phi_{t}$ has also average $0,$ as $\Phi_{t}$ preserves the volume.\\
This justifies our choice of a symplectic integrator, as it implies that at each step of the scheme \eqref{omega-n}, $\omega_{n}=\mathcal{S}_{\tau}^{n}(\omega_{0})$ has average $0,$ and we may define the divergence-free vector field $J\nabla\Delta^{-1}\omega_{n},$ and thus compute $\omega_{n+1},$ and so on.
\end{remark}

\subsection{Stability estimates}

Our analysis of the stability of the semi-discrete operator defined by \eqref{SD-OP} is based on the fact that the implicit midpoint rule is the composition of Euler's backward and forward methods, with half time-steps, as it was shown in the third point of Proposition \ref{Mid-topo}.\\
Therefore, to control the regularity (in space) of some function $g\circ \Phi_{t},$ we shall first analyse the effect of $\Ec_{t}$ (Lemma \ref{EulerExp} below), and then the effect of $\Ecs_{t}$ (Lemma \ref{EulerImp} below).

\begin{lemma}
\label{EulerExp}
Let $s\geq 3,$ $\omega_{0} \in H^{s}(\mathbbm{T}^{2})$ with average $0,$ and $\tau \in ]0,1[.$ Let $\psi$ be the solution of the Poisson equation $\Delta \psi =\omega_{0},$ and let $\Ec_{t}$ be the half time-step forward Euler integrator defined for $t\in [0,\tau]$ by the formula
$$\Ec_{t}(x)=x+\frac{t}{2}J\nabla\psi(x).$$
There exists two positive constants $R_{0}$ and $R_{1},$ independent of $\omega_{0},$ such that, if $\tau \left\| \omega_{0}\right\|_{H^{2}(\mathbbm{T}^{2})}R_{0}<1,$ then for all $g\in H^{s}(\mathbbm{T}^{2})$ and all $t\in [0,\tau],$
$$\left\| g\circ \Ec_{t}\right\|_{H^{s}(\mathbbm{T}^{2})}\leq e^{R_{1}t \left(1+t\left\| \omega_{0} \right\|_{H^{s}(\mathbbm{T}^{2})}\right) \left\| \omega_{0} \right\|_{H^{s}(\mathbbm{T}^{2})}}\left\| g\right\|_{H^{s}(\mathbbm{T}^{2})}.$$
\end{lemma}

\begin{proof}
The idea is to derive a transport equation whose initial data is $g$ and whose final data is $g\circ \Ec_{t},$ and to obtain the conclusion by Lemma \ref{stability}.\\
Let us consider the transport equation
\begin{equation}
\label{transEe}
\left\{
\begin{split}
&\partial_{t} r(t,x)- \rX(t,x) \cdot\nabla r(t,x)=0 \\
&r(0,x)=g(x),
\end{split}
\right.
\end{equation}
with 
\begin{equation}
\label{chant}
\rX(t,x)=\frac{1}{2} \left(I_{2}+\frac{t}{2}J\nabla^{2}\psi(x)\right)^{-1}J\nabla\psi(x).
\end{equation}
Note that the inversion of the above matrix has already been justified in the proof of Proposition \ref{Mid-topo}, under the hypothesis $\tau \left\| \omega_{0}\right\|_{H^{2}(\mathbbm{T}^{2})}R_{0}<1.$ As
$$\Ecs_{-t}(x)=x-\frac{t}{2}J\nabla \psi(\Ecs_{-t}(x)),$$
we have
$$\partial_{t}(\Ecs_{-t}(x)) = -\rX(t,\Ecs_{-t}(x)),$$
such that for all $(t,x)\in [0,\tau]\times \mathbbm{T}^{2},$
$$\frac{\dd}{\dd t} r(t,\Ecs_{-t}(x))=0,$$
and thus for all $(t,x)\in [0,\tau]\times \mathbbm{T}^{2},$
$$r(t,\Ecs_{-t}(x))=g(x).$$
In other words, as $\Ecs_{-t}(x)=\Ec_{t}^{-1}(x),$ we have
$$r(t,x)=g\circ \Ec_{t}(x).$$
That being said, using estimate \eqref{EE3} from Lemma \ref{stability} (with the inequality $s\geq 3$), we may write that for all $t\in [0,\tau],$
$$\frac{\dd}{\dd t}\left\| r(t,\cdot)\right\|_{H^{s}(\mathbbm{T}^{2})}^{2}\leq  C \left\| \rX(t,\cdot)\right\|_{H^{s}(\mathbbm{T}^{2})} \left\|r(t,\cdot)\right\|_{H^{s}(\mathbbm{T}^{2})}^{2}+ \left\| \nabla \cdot \rX(t,\cdot)\right\|_{L^{\infty}(\mathbbm{T}^{2})}\left\|r(t,\cdot)\right\|_{H^{s}(\mathbbm{T}^{2})}^{2}.$$
Using Lemma \ref{fieldSD} below and Gronwall's Lemma, we infer that
$$\left\| r(t,\cdot)\right\|_{H^{s}(\mathbbm{T}^{2})} \leq \left\| r(0,\cdot)\right\|_{H^{s}(\mathbbm{T}^{2})} e^{R_{1}t \left(1+t\left\| \omega_{0} \right\|_{H^{s}(\mathbbm{T}^{2})}\right) \left\| \omega_{0} \right\|_{H^{s}(\mathbbm{T}^{2})}},$$
which gives the desired conclusion, as $r(0)=g$ and $r(t)=g\circ \Ec_{t}.$
\end{proof}

\begin{lemma}
\label{fieldSD}
Let $s\geq 3.$ For $t\in[-\tau,\tau],$ with $\tau$ satisfying the hypothesis $\tau R_{0} \left\| \omega_{0}\right\|_{H^{2}(\mathbbm{T}^{2})}<1$ of Proposition \ref{Mid-topo} and Lemma \ref{EulerExp}, and $x\in\mathbbm{T}^{2},$ let us consider the vector field $\rX(t,x)$ defined by \eqref{chant}.
There exists a constant $C>0$ such that
\begin{equation}
\label{regularite}
\left\| \rX(t,\cdot)\right\|_{H^{s}(\mathbbm{T}^{2})} \leq C \left( 1 + |t|\left\| \omega_{0}\right\|_{H^{s}(\mathbbm{T}^{2})}\right)\left\| \omega_{0}\right\|_{H^{s}(\mathbbm{T}^{2})}
\end{equation}
and
\begin{equation}
\label{divergence}
\left\| \nabla \cdot \rX(t,\cdot)\right\|_{L^{\infty}(\mathbbm{T}^{2})} \leq C \left( 1 + |t|\left\| \omega_{0}\right\|_{H^{s}(\mathbbm{T}^{2})}\right)\left\| \omega_{0}\right\|_{H^{s}(\mathbbm{T}^{2})}.
\end{equation}
\end{lemma}

\begin{Proof}
From definition \eqref{chant}, we may write that
$$2\rX(t,x)=J\nabla \psi (x) + F(tJ\nabla^{2} \psi(x)) J\nabla \psi(x),$$
where $F:\mathcal{U}\subset M_{2}(\mathbbm{R}) \to M_{2}(\mathbbm{R})$ is a smooth function defined on a sufficiently small neighborhood $\mathcal{U}$ of $0_{2},$ the zero element of the vector space $M_{2}(\mathbbm{R}),$ by the formula
$$F(A)= \left(I_{2} + \frac{1}{2}A\right)^{-1} - I_{2} = \sum_{n=1}^{\infty}\frac{(-1)^{n}}{2^{n}}A^{n}.$$
Since (see \eqref{reg})
$$|t| \left\|J\nabla^{2} \psi\right\|_{L^{\infty}(\mathbbm{T}^{2})}\leq C \tau \left\| \omega_{0}\right\|_{H^{2}(\mathbbm{T}^{2})},$$
we way assume, up to a proper modification of $R_{0},$ that $tJ\nabla^{2} \psi$ belongs to $\mathcal{U}$ almost everywhere. Applying Lemma \ref{Kato}, we infer that
$$\left\| \rX(t,\cdot)\right\|_{H^{s}(\mathbbm{T}^{2})}\leq \left\| J\nabla \psi\right\|_{H^{s}(\mathbbm{T}^{2})} + C_{s}\left(\left\| tJ\nabla^{2} \psi\right\|_{L^{\infty}(\mathbbm{T}^{2})}\right) \left( 1 + \left\| tJ\nabla^{2} \psi\right\|_{H^{s}(\mathbbm{T}^{2})}\right)\left\| J\nabla \psi\right\|_{H^{s}(\mathbbm{T}^{2})},$$
where $C_{s}:\mathbbm{R}_{+}\to \mathbbm{R}_{+}$ is an increasing continuous function. In view of the hypothesis $\tau R_{0} \left\| \omega_{0}\right\|_{H^{2}(\mathbbm{T}^{2})}<1$ and of estimate \eqref{reg}, we may assume that for all $|t|\leq \tau,$
$$C_{s}\left(\left\| tJ\nabla^{2} \psi\right\|_{L^{\infty}(\mathbbm{T}^{2})}\right)\leq C,$$
for some appropriate constant $C.$ Thus we obtain the estimate
$$\left\| \rX(t,\cdot)\right\|_{H^{s}(\mathbbm{T}^{2})}\leq C \left( 1 + |t|\left\| \omega_{0}\right\|_{H^{s}(\mathbbm{T}^{2})}\right)\left\| \omega_{0}\right\|_{H^{s}(\mathbbm{T}^{2})},$$
which is precisely estimate \eqref{regularite}.\\
Estimate \eqref{divergence} follows quickly using the Sobolev embedding \eqref{Sob-inj2} as follows
$$\left\| \nabla \cdot \rX(t,\cdot)\right\|_{L^{\infty}(\mathbbm{T}^{2})} \leq C \left\|\rX(t,\cdot) \right\|_{H^{3}(\mathbbm{T}^{2})}\leq \left\| \rX(t,\cdot)\right\|_{H^{s}(\mathbbm{T}^{2})} \leq C \left( 1 + |t|\left\| \omega_{0}\right\|_{H^{s}(\mathbbm{T}^{2})}\right)\left\| \omega_{0}\right\|_{H^{s}(\mathbbm{T}^{2})},$$
since $s\geq 3.$
\end{Proof}

\begin{lemma}
\label{EulerImp}
Let  $s\geq 3,$ $\omega_{0} \in H^{s}(\mathbbm{T}^{2})$ with average $0,$ and $\tau \in ]0,1[.$  Let $\psi$ be the solution of the Poisson equation $\Delta \psi =\omega_{0},$ and let $\Ecs_{t}$ be the half time-step backward Euler integrator defined for $t\in [0,\tau]$ by the formula
$$\Ecs_{t}(x)= x+\frac{t}{2}J\nabla\psi(\Ecs_{t}(x)).$$
There exists two positive constants $R_{0},R_{1},$ independent of $\omega_{0},$ such that, if $\tau R_{0}\left\| \omega_{0}\right\|_{H^{2}(\mathbbm{T}^{2})}<1,$ then for all $g\in H^{s}(\mathbbm{T}^{2})$  and $t\in [0,\tau],$
$$\left\| g\circ \Ecs_{t} \right\|_{H^{s}(\mathbbm{T}^{2})} \leq e^{R_{1}t \left(1+t\left\| \omega_{0} \right\|_{H^{s}(\mathbbm{T}^{2})}\right) \left\| \omega_{0} \right\|_{H^{s}(\mathbbm{T}^{2})}} \left\| g \right\|_{H^{s}(\mathbbm{T}^{2})}.$$
\end{lemma}

\begin{proof}
Our proof ressembles the proof of Lemma \ref{EulerExp}, as we take advantage of the fact that one travels from $g\circ \Ecs_{t}$ to $g$ along the flow $\Ec_{-t}=\Ecs_{t}\hspace{0mm}^{-1} .$ Thus, instead of deriving a new equation that will transport us from $g$ to $g\circ \Ecs_{t},$ we shall use once more equation \eqref{transEe} (with the time reversed, essentially), with this time $g\circ \Ecs_{t}$ for initial data, and $g$ for final data, and we shall conclude by Lemma \ref{stability}.\\
Let us indeed consider the transport equation
\begin{equation}
\label{transEi}
\left\{
\begin{split}
&\partial_{\sigma}r(\sigma,t,x) + \rX(-\sigma,x)\cdot\nabla r(\sigma,t,x)=0\\
&r(0,t,x) = g\circ \Ecs_{t}(x),
\end{split}
\right.
\end{equation}
where the auxiliary variable $\sigma$ belongs to $[0,t],$ and where $\rX$ was defined by \eqref{chant}. Note that, as
$$\Ecs_{\sigma}(x)=x+\frac{\sigma}{2}J\nabla \psi(\Ecs_{\sigma}(x)),$$
we have the identity 
$$\partial_{\sigma}\Ecs_{\sigma}(x)=\rX(-\sigma,\Ecs_{\sigma}(x)).$$
Therefore,
$$\frac{\dd}{\dd \sigma} r(\sigma, t, \Ecs_{\sigma}(x))=0,$$
such that for all $\sigma\in [0,t]$ and $x\in\mathbbm{T}^{2},$
$$r(\sigma,t,\Ecs_{\sigma}(x))=r(0,t,x)=g\circ \Ecs_{t}(x),$$
and thus
$$r(t,t,x)= g\circ\Ecs_{t}\circ \Ec_{-t}(x)=g(x).$$
Therefore the transport equation \eqref{transEi} has the expected initial and final data. However, it will be more convenient to deal with the function
$$u(\sigma,t,x)=r(-\sigma,t,x),$$
with $\sigma\in[-t,0].$ (Essentially, we do this to apply Gronwall's Lemma in its usual statement at the end of the proof.)\\
$u$ satisfies on $[-t,0]$ the transport equation
\begin{equation}
\label{transEi2}
\left\{
\begin{split}
&\partial_{\sigma}u(\sigma,t,x) -\rX(\sigma,x)\cdot\nabla u(\sigma,t,x)=0\\
&u(0,t,x) = r(0,t,x)=g\circ \Ecs_{t}(x),\\
&u(-t,t,x) = r(t,t,x)= g(x).
\end{split}
\right.
\end{equation}
That being said, estimate \eqref{EE3} from Lemma \ref{stability} shows that for any $\sigma\in [-t,0],$
$$\frac{\dd}{\dd \sigma} \left\| u(\sigma,t,\cdot)\right\|_{H^{s}(\mathbbm{T}^{2})}^{2} \leq \left\| \rX(\sigma,\cdot)\right\|_{H^{s}(\mathbbm{T}^{2})} \left\| u(\sigma,t,\cdot)\right\|_{H^{s}(\mathbbm{T}^{2})}^{2} + \left\| \nabla \cdot \rX(\sigma,\cdot)\right\|_{L^{\infty}(\mathbbm{T}^{2})} \left\| u(\sigma,t,\cdot)\right\|_{H^{s}(\mathbbm{T}^{2})}^{2}.$$
Gronwall's Lemma and Lemma \ref{fieldSD} above imply then that
$$\left\| u(\sigma,t,\cdot)\right\|_{H^{s}(\mathbbm{T}^{2})}^{2}\leq \left\| u(-t,t,\cdot) \right\|_{H^{s}(\mathbbm{T}^{2})}^{2} \mbox{exp}\left(\int_{-t}^{\sigma} 2C \left(1+|\theta| \left\| \omega_{0} \right\|_{H^{s}(\mathbbm{T}^{2})}\right) \left\| \omega_{0} \right\|_{H^{s}(\mathbbm{T}^{2})} \dd \theta\right).$$
Taking $\sigma=0$ gives us the estimate
$$\left\| u(0,t,\cdot)\right\|_{H^{s}(\mathbbm{T}^{2})}\leq \left\| u(-t,t,\cdot) \right\|_{H^{s}(\mathbbm{T}^{2})} e^{R_{1}t \left(1+t\left\| \omega_{0} \right\|_{H^{s}(\mathbbm{T}^{2})}\right) \left\| \omega_{0} \right\|_{H^{s}(\mathbbm{T}^{2})}} ,$$
which gives the desired conclusion, using the second and third lines of \eqref{transEi2}.
\end{proof}

\begin{proposition}
\label{numstab}
Let  $s\geq 3,$ $\omega_{0} \in H^{s}(\mathbbm{T}^{2})$ with average $0,$ and $\tau \in ]0,1[.$ There exists two positive constants $R_{0},R_{1},$ independent of $\omega_{0},$ such that, if $\tau R_{0}\left\| \omega_{0}\right\|_{H^{2}(\mathbbm{T}^{2})}<1,$ then for all $t\in [0,\tau],$
$$\left\| \mathcal{S}_{t}(\omega_{0})\right\|_{H^{s}(\mathbbm{T}^{2})} \leq e^{R_{1}t \left(1+t\left\| \omega_{0} \right\|_{H^{s}(\mathbbm{T}^{2})}\right) \left\| \omega_{0} \right\|_{H^{s}(\mathbbm{T}^{2})}} \left\| \omega_{0}\right\|_{H^{s}(\mathbbm{T}^{2})},$$
where the operator $\mathcal{S}_{t}$ is defined by formula \eqref{SD-OP}.
\end{proposition}

\begin{proof}
As already seen, we may write
$$\mathcal{S}_{t}(\omega_{0})= \omega_{0} \circ \Phi_{t}=\omega_{0} \circ \Ec_{t} \circ \Ecs_{t},$$
with $\Ecs_{t}$ and $\Ec_{t}$ defined in Proposition \ref{Mid-topo}. We shall apply Lemma \ref{EulerImp} and \ref{EulerExp}. However, note that in these Lemmas, we derived a bound for $g\circ \Ecs_{t}$ (or $g\circ\Ec_{t}$) where $g$ is a function that depends only on $x.$ Hence, to apply these Lemmas, we shall consider the function
$$f(\sigma,t,x)=\omega_{0}\circ \Ec_{\sigma} \circ \Ecs_{t}(x),$$
with $t,\sigma, \in[0,\tau].$\\
In view of the hypothesis $\tau R_{0}\left\| \omega_{0}\right\|_{H^{2}(\mathbbm{T}^{2})}<1,$ we may apply Lemma \ref{EulerImp}, which shows that for all $t,\sigma\in [0,\tau],$
$$\left\| f(\sigma,t,\cdot)\right\|_{H^{s}(\mathbbm{T}^{2})}\leq  e^{Ct \left(1+t\left\| \omega_{0} \right\|_{H^{s}(\mathbbm{T}^{2})}\right) \left\| \omega_{0} \right\|_{H^{s}(\mathbbm{T}^{2})}} \left\| \omega_{0} \circ \Ec_{\sigma}\right\|_{H^{s}(\mathbbm{T}^{2})},$$
for some constant $C>0.$\\
Applying now Lemma \ref{EulerExp}, we may also write that for all $\sigma \in [0,\tau],$
$$\left\| \omega_{0} \circ \Ec_{\sigma}\right\|_{H^{s}(\mathbbm{T}^{2})}\leq e^{R_{1}t \left(1+\sigma\left\| \omega_{0} \right\|_{H^{s}(\mathbbm{T}^{2})}\right) \left\| \omega_{0} \right\|_{H^{s}(\mathbbm{T}^{2})}} \left\| \omega_{0}\right\|_{H^{s}(\mathbbm{T}^{2})}.$$
Hence, for all $t,\sigma\in [0,\tau],$
$$\left\| f(\sigma,t,\cdot)\right\|_{H^{s}(\mathbbm{T}^{2})}\leq e^{Ct \left(2+(t+\sigma)\left\| \omega_{0} \right\|_{H^{s}(\mathbbm{T}^{2})}\right) \left\| \omega_{0} \right\|_{H^{s}(\mathbbm{T}^{2})}} \left\| \omega_{0}\right\|_{H^{s}(\mathbbm{T}^{2})}.$$
This gives the result by taking $\sigma=t,$ and $R_{1}=2C.$
\end{proof}

\begin{corollary}
\label{numstab2}
Let $s\geq 5,$ and $u,v \in H^{s}(\mathbbm{T}^{2})$ with average $0.$ Let $\tau \in ]0,1[.$ There exists two positive constants $R_{0},R_{1},$ independent of $u$ and $v,$ such that, if
$$\tau R_{0} \max\left(\left\| u \right\|_{H^{2}(\mathbbm{T}^{2})},\left\| v \right\|_{H^{s-3}(\mathbbm{T}^{2})}\right)<1,$$
for all $t\in [0,\tau],$
$$\left\|  \mathcal{S}_{t}(u)-\mathcal{S}_{t}(v)\right\|_{H^{s-4}(\mathbbm{T}^{2})}\leq \left\| u-v\right\|_{H^{s-4}(\mathbbm{T}^{2})} e^{\tau R_{1}\left(1+ \left\| u \right\|_{H^{s-3}(\mathbbm{T}^{2})}\right)} + R_{1} \tau^{3} \left\| u \right\|_{H^{s-3}(\mathbbm{T}^{2})}.$$
where the operator $\mathcal{S}_{t}$ is defined by formula \eqref{SD-OP}.
\end{corollary}

\begin{proof}
We may apply Proposition \ref{Mid-topo} and define on $[0,\tau]$ the midpoint integrator associated with $u,$ namely
$$\Phi_{t}(x)=x+tJ\nabla \Delta^{-1}u\left(\frac{x+\Phi^{t}(x)}{2}\right).$$
As
$$\mathcal{S}_{t}(u)=u\circ \Phi_{t},$$
we have 
$$\frac{\dd}{\dd t} \left[\mathcal{S}_{t}(u)\circ \Phi_{-t} \right] = 0,$$
such that $\mathcal{S}_{t}(u)$ satisfies the transport equation
\begin{equation}
\notag
\left\{
\begin{split}
&\partial_{t} \mathcal{S}_{t}(u) + V_{u}(t)\cdot \nabla \mathcal{S}_{t}(u) =0,\\
& V_{u}(t,x)= \partial_{t} \Phi_{-t} \circ\Phi_{t}(x).
\end{split}
\right.
\end{equation}
Moreover, assertions {\bf ii)} and {\bf iv)} from Proposition \ref{Mid-topo} show that $V_{u}$ has $C^{s-2}$ regularity, and that there exists a $C^{s-4}$ vector field $R_{u}(t,x)$ such that for all $(t,x)\in [0,\tau]\times \mathbbm{T}^{2},$
$$V_{u}(t,x)= -J\nabla \Delta^{-1} u + t^{2} R_{u}(t,x).$$
With the same arguments, there exists a $C^{s-2}$ vector field $V_{v}(t,x)$ such that 
\begin{equation}
\notag
\left\{
\begin{split}
&\partial_{t} \mathcal{S}_{t}(v) + V_{v}(t)\cdot \nabla \mathcal{S}_{t}(v) =0,\\
& V_{v}(t,x)= -J\nabla \Delta^{-1} v + t^{2} R_{v}(t,x).
\end{split}
\right.
\end{equation}
Moreover $R_{v}(t,x)$ has $C^{s-4}$ regularity. Therefore $\mathcal{S}_{t}(u)-\mathcal{S}_{t}(v)$ solves the equation
$$\partial_{t}(\mathcal{S}_{t}(u)-\mathcal{S}_{t}(v)) + V_{u}(t) \cdot \nabla \left(\mathcal{S}_{t}(u)-\mathcal{S}_{t}(v) \right)= \left(V_{v}(t)-V_{u}(t)\right)\cdot \nabla \mathcal{S}_{t}(v).$$
Applying estimate \eqref{EE2} from Lemma \ref{stability}, we infer that 
\begin{equation}
\notag
\begin{split}
\frac{\dd}{\dd t} \left\|  \mathcal{S}_{t}(u)-\mathcal{S}_{t}(v)\right\|_{H^{s-4}(\mathbbm{T}^{2})}^{2} &\leq C\left\| V_{u}(t) \right\|_{H^{s-2}(\mathbbm{T}^{2})}  \left\|  \mathcal{S}_{t}(u)-\mathcal{S}_{t}(v)\right\|_{H^{s-4}(\mathbbm{T}^{2})}^{2} \\
& + \left\| \nabla \cdot V_{u}(t) \right\|_{L^{\infty}(\mathbbm{T}^{2})}  \left\|  \mathcal{S}_{t}(u)-\mathcal{S}_{t}(v)\right\|_{H^{s-4}(\mathbbm{T}^{2})}^{2}\\
& + 2 \left\| \left(V_{v}(t)-V_{u}(t)\right)\cdot \nabla \mathcal{S}_{t}(v)\right\|_{H^{s-4}(\mathbbm{T}^{2})} \left\|  \mathcal{S}_{t}(u)-\mathcal{S}_{t}(v)\right\|_{H^{s-4}(\mathbbm{T}^{2})}.
\end{split}
\end{equation}
Since $V_{u}(t)$ has class $C^{s-2},$ with $s\geq 5,$ we may write that for all $t\in [0,\tau],$
$$\left\| V_{u}(t) \right\|_{H^{s-2}(\mathbbm{T}^{2})}\leq C \quad \mbox{and} \quad \left\| \nabla \cdot V_{u}(t) \right\|_{L^{\infty}(\mathbbm{T}^{2})}  \leq C.$$
Also, we may find a $C^{s-4}$ vector field $R(t,x)$ such that
$$V_{v}(t)-V_{u}(t) = J\nabla \Delta^{-1} u - J\nabla \Delta^{-1}v + t^{2} R(t,x).$$
Using estimate \eqref{transport2} from Lemma \ref{transport}, this implies that
\begin{multline*}
\left\| \left(V_{v}(t)-V_{u}(t)\right)\cdot \nabla \mathcal{S}_{t}(v)\right\|_{H^{s-4}(\mathbbm{T}^{2})} \leq C \left(\left\|J\nabla \Delta^{-1} u - J\nabla \Delta^{-1}v \right\|_{H^{s-3}(\mathbbm{T}^{2})} + t^{2}\right) \left\| \mathcal{S}_{t}(v) \right\|_{H^{s-3}(\mathbbm{T}^{2})} \\
\leq C \left(\left\| u - v \right\|_{H^{s-4}(\mathbbm{T}^{2})} + t^{2}\right) \left\| \mathcal{S}_{t}(v) \right\|_{H^{s-3}(\mathbbm{T}^{2})}.
\end{multline*}
Using Proposition \ref{numstab}, we may write, under the hypothesis $\tau R_{0} \left\| v\right\|_{H^{s-3}(\mathbbm{T}^{2})}<1,$ 
$$\left\| \mathcal{S}_{t}(v) \right\|_{H^{s-3}(\mathbbm{T}^{2})} \leq C\left\| v\right\|_{H^{s-3}(\mathbbm{T}^{2})}.$$
Collecting the previous estimates, and applying Lemma \ref{inequadiff}, we infer that
\begin{equation}
\notag
\begin{split}
\left\|  \mathcal{S}_{t}(u)-\mathcal{S}_{t}(v)\right\|_{H^{s-4}(\mathbbm{T}^{2})} &\leq \left\| u-v\right\|_{H^{s-4}(\mathbbm{T}^{2})} +  \int_{0}^{t}C \left\|  \mathcal{S}_{\sigma}(u)-\mathcal{S}_{\sigma}(v)\right\|_{H^{s-4}(\mathbbm{T}^{2})} \dd \sigma\\
& + C \tau\left\| u \right\|_{H^{s-3}(\mathbbm{T}^{2})} \left( \left\| u-v\right\|_{H^{s-4}(\mathbbm{T}^{2})} +\tau^{2}\right).
\end{split}
\end{equation}
Applying Gronwall's Lemma we conclude that
$$\left\|  \mathcal{S}_{t}(u)-\mathcal{S}_{t}(v)\right\|_{H^{s-4}(\mathbbm{T}^{2})}\leq \left\| u-v\right\|_{H^{s-4}(\mathbbm{T}^{2})} (1+ C \tau\left\| u \right\|_{H^{s-3}(\mathbbm{T}^{2})}) e^{\tau C} + C \tau^{3} \left\| u \right\|_{H^{s-3}(\mathbbm{T}^{2})}e^{C\tau}.$$
One obtains then the desired conclusion with the inequality
$$1+x\leq e^{x}$$
that holds for any $x\geq 0,$ and with an appropriate choice of the constant $R_{1}=R_{1}(C).$
\end{proof}

\section{Convergence estimates}

\subsection{Local errors}

\begin{proposition}
\label{error1}
Let $s\geq 2$ and $\omega_{0}\in H^{s}(\mathbbm{T}^{2})$ with average $0.$ There exists two positive constants $R_{0}, R_{1}$ independent of $\omega_{0},$ such that, if $\tau R_{0} \left\| \omega_{0} \right\|_{H^{s}(\mathbbm{T}^{2})} <1,$ then for all $t\in [0,\tau],$
$$\left\| \varphi_{E,t}(\omega_{0}) - \varphi_{F,t}(\omega_{0}) \right\|_{H^{s-1}(\mathbbm{T}^{2})}\leq R_{1}\tau^{2}\left\| \omega_{0} \right\|_{H^{s}(\mathbbm{T}^{2})}^{3}.$$
\end{proposition}

\begin{proof}
$\varphi_{E,t}(\omega_{0})$ and $\varphi_{F,t}(\omega_{0})$ satisfy respectively the transport equations
$$\partial_{t} \varphi_{E,t}(\omega_{0}) - J\nabla \Delta^{-1} \varphi_{E,t}(\omega_{0})\cdot \nabla \varphi_{E,t}(\omega_{0})=0$$
and
$$\partial_{t} \varphi_{F,t}(\omega_{0}) - J\nabla \Delta^{-1} \omega_{0}\cdot \nabla \varphi_{F,t}(\omega_{0}) =0 ,$$
with initial data $\omega_{0}.$\\
Hence
\begin{multline*}
\partial_{t}(\varphi_{E,t}(\omega_{0})-\varphi_{F,t}(\omega_{0})) - J\nabla\Delta^{-1} \varphi_{E,t}(\omega_{0})\cdot \nabla\left(\varphi_{E,t}(\omega_{0})-\varphi_{F,t}(\omega_{0})\right)\\
=J\nabla\left(\Delta^{-1}\omega_{0} - \Delta^{-1} \varphi_{E,t}(\omega_{0})\right)\cdot \nabla\varphi_{F,t}(\omega_{0}).
\end{multline*}
Therefore, using estimate \eqref{EE2} from Lemma \ref{stability}, applied with the divergence-free vector field
$$X(t)=J\nabla\Delta^{-1}\varphi_{E,t}(\omega_{0}),$$
which satisfies (using Proposition \ref{poisson-reg})
$$\left\| J\nabla\Delta^{-1}\varphi_{E,t}(\omega_{0}) \right\|_{H^{s+1}(\mathbbm{T}^{2})} \leq C \left\| \varphi_{E,t}(\omega_{0}) \right\|_{H^{s}(\mathbbm{T}^{2})}.$$
we obtain the estimate
\begin{equation}
\notag
\begin{split}
&\frac{\dd}{\dd t} \left\| \varphi_{E,t}(\omega_{0})- \varphi_{F,t}(\omega_{0})\right\|_{H^{s-1}(\mathbbm{T}^{2})}^{2} \leq  C\left\| \varphi_{E,t}(\omega_{0})-\varphi_{F,t}(\omega_{0}) \right\|_{H^{s-1}(\mathbbm{T}^{2})}^{2} \left\| \varphi_{E,t}(\omega_{0})\right\|_{H^{s}(\mathbbm{T}^{2})} \\
&\hspace{30mm}+ 2\left\| J\nabla \left( \Delta ^{-1}\omega_{0}-\Delta^{-1}\varphi_{E,t}(\omega_{0})\right)\cdot \nabla\varphi_{F,t}(\omega_{0}) \right\|_{H^{s-1}(\mathbbm{T}^{2})} \left\| \varphi_{E,t}(\omega_{0})-\varphi_{F,t}(\omega_{0}) \right\|_{H^{s-1}(\mathbbm{T}^{2})}.
\end{split}
\end{equation}
In view of the hypothesis $\tau R_{0} \left\| \omega_{0}\right\|_{H^{s}(\mathbbm{T}^{2})}<1$ we may apply Proposition \ref{stab}, which shows that
$$\left\| \varphi_{E,t}(\omega_{0})\right\|_{H^{s}(\mathbbm{T}^{2})} \leq C\left\| \omega_{0}\right\|_{H^{s}(\mathbbm{T}^{2})}.$$
Also, using estimate \eqref{transport2} form Lemma \ref{transport}, and Propositions \ref{stab} and \ref{poisson-reg}, we have
\begin{multline*}
\left\| J\nabla \left( \Delta ^{-1}\omega_{0}-\Delta^{-1}\varphi_{E,t}(\omega_{0})\right) \cdot \nabla\varphi_{F,t}(\omega_{0}) \right\|_{H^{s-1}(\mathbbm{T}^{2})} \leq C \left\| \Delta^{-1}(\omega_{0}- \varphi_{E,t}(\omega_{0})) \right\|_{H^{s+1}(\mathbbm{T}^{2})}\left\| \varphi_{F,t}(\omega_{0})\right\|_{H^{s}(\mathbbm{T}^{2})} \\
\leq C \left\| \omega_{0} - \varphi_{E,t}(\omega_{0}) \right\|_{H^{s-1}(\mathbbm{T}^{2})} \left\| \omega_{0}\right\|_{H^{s}(\mathbbm{T}^{2})}
\end{multline*}
However, using once more Proposition \ref{stab} and estimate \eqref{transport2} from Lemma \ref{transport},
\begin{multline*}
\left\| \varphi_{E,t}(\omega_{0})-\omega_{0} \right\|_{H^{s-1}(\mathbbm{T}^{2})}\leq \int_{0}^{t}\left\| J\nabla \Delta^{-1} \varphi_{E,\sigma}(\omega_{0}) \cdot \nabla \varphi_{E,\sigma}(\omega_{0}) \right\|_{H^{s-1}(\mathbbm{T}^{2})}\dd \sigma \\
\leq \int_{0}^{t} \left\| \varphi_{E,\sigma}(\omega_{0})\right\|_{H^{s}(\mathbbm{T}^{2})}^{2}\leq Ct \left\| \omega_{0} \right\|_{H^{s}(\mathbbm{T}^{2})}^{2} .
\end{multline*}
Collecting the previous estimates, we infer that
\begin{multline*}
\frac{\dd}{\dd t} \left\| \varphi_{E,t}(\omega_{0})- \varphi_{F,t}(\omega_{0})\right\|_{H^{s-1}(\mathbbm{T}^{2})}^{2} \leq C\left\| \omega_{0} \right\|_{H^{s}(\mathbbm{T}^{2})}\left\| \varphi_{E,t}(\omega_{0})-\varphi_{F,t}(\omega_{0}) \right\|_{H^{s-1}(\mathbbm{T}^{2})}^{2} \\
+C\left\| \omega_{0} \right\|_{H^{s}(\mathbbm{T}^{2})}^{3}t \left\| \varphi_{E,t}(\omega_{0})-\varphi_{F,t}(\omega_{0}) \right\|_{H^{s-1}(\mathbbm{T}^{2})}.
\end{multline*}
Using Lemma \ref{inequadiff} below, we conclude that
\begin{multline*}
\left\| \varphi_{E,t}(\omega_{0})- \varphi_{F,t}(\omega_{0})\right\|_{H^{s-1}(\mathbbm{T}^{2})}\leq C\left\| \omega_{0} \right\|_{H^{s}(\mathbbm{T}^{2})} \int_{0}^{t} \left\| \varphi_{E,\sigma}(\omega_{0})-\varphi_{F,\sigma}(\omega_{0})\right\|_{H^{s-1}(\mathbbm{T}^{2})}\dd \sigma\\
+ C\left\| \omega_{0} \right\|_{H^{s}(\mathbbm{T}^{2})}^{3}\int_{0}^{t} \sigma \dd \sigma.
\end{multline*}
Thus, if $\tau C\left\| \omega_{0} \right\|_{H^{s}(\mathbbm{T}^{2})}<1,$ which we may assume, choosing if necessary $R_{0} >C,$ we have
$$\sup_{t\in[0,\tau]} \left\| \varphi_{E,t}(\omega_{0})- \varphi_{F,t}(\omega_{0})\right\|_{H^{s-1}(\mathbbm{T}^{2})}\leq R_{1}\left\| \omega_{0} \right\|_{H^{s}(\mathbbm{T}^{2})}^{3} \tau^{2},$$
for some appropriate constant $R_{1}(C).$
\end{proof}

The previous proof uses the following result, inspired by Lemma $2.9$ of \cite{F-G}.

\begin{lemma} 
\label{inequadiff}
Let $f:\mathbbm{R}\to \mathbbm{R}_{+}$ be a continuous function, and $y: \mathbbm{R}\to \mathbbm{R}_{+} $be a differentiable function satisfying the inequality 
$$\forall t \in\mathbbm{R}, \quad \frac{\dd}{\dd t} y(t)\leq 2C_{1}y(t) + 2C_{2} \sqrt{y(t)}f(t),$$
where $C_{1}$ and $C_{2}$ are two positive constants. Then
$$\forall t \in\mathbbm{R}, \quad \sqrt{y(t)}\leq \sqrt{y(0)} + C_{1}\int_{0}^{t} \sqrt{y(\sigma)} \dd \sigma + C_{2} \int_{0}^{t}f(\sigma) \dd \sigma.$$
\end{lemma}

\begin{proof}
For $\eps >0,$ we define $y_{\eps}=y+\eps.$ We have then
$$\frac{\dd }{\dd t} \sqrt{y_{\eps}(t)}=\frac{1}{2\sqrt{y_{\eps}(t)}} \frac{\dd} {\dd t} y(t) \leq \frac{C_{1} y(t)}{\sqrt{y_{\eps}(t)}} + \frac{C_{2}\sqrt{y(t)}}{\sqrt{y_{\eps}(t)}} f(t).$$
Therefore
$$\sqrt{y_{\eps}(t)}\leq \sqrt{y_{\eps}(0)} + C_{1}\int_{0}^{t} \frac{ y(\sigma)}{\sqrt{y_{\eps}(\sigma)}}\dd \sigma + C_{2}\int_{0}^{t} \frac{\sqrt{y(\sigma)}}{\sqrt{y_{\eps}(\sigma)}} f(\sigma)\dd \sigma.$$
Taking then the limit $\eps\to 0$ proves the Lemma.
\end{proof}

\begin{proposition}
\label{error2}
Let $s\geq 5$ and $\omega_{0}\in H^{s}(\mathbbm{T}^{2}),$ with average $0.$ There exists two positive constants $R_{0}$ and $R_{1},$ independent of $\omega_{0},$ such that, if $\tau \left\| \omega_{0} \right\|_{H^{s-3}(\mathbbm{T}^{2})}R_{0}<1,$ then for all $t\in [0,\tau],$
$$\left\| \varphi_{F,t}(\omega_{0}) - \mathcal{S}_{t}(\omega_{0})\right\|_{H^{s-4}(\mathbbm{T}^{2})} \leq R_{1} \tau^{3} \left\| \omega_{0}\right\|_{H^{s-3}(\mathbbm{T}^{2})} .$$
\end{proposition}

\begin{proof}
In view of the hypothesis $\tau \left\| \omega_{0} \right\|_{H^{s-3}(\mathbbm{T}^{2})}R_{0}<1$ and $s\geq 5,$ we may assume that we are in the frame of Propositions \ref{Mid-topo} and \ref{numstab}, and thus that their respective conclusions hold.\\
That being said, it was shown in the proof of Corollary \ref{numstab2} that there exists a $C^{s-2}$ vector field $V(t,x)$ such that $\mathcal{S}_{t}(\omega_{0})$ solves the equation
$$\partial_{t} \mathcal{S}_{t}(\omega_{0}) + V(t,\cdot) \cdot \nabla \mathcal{S}_{t}(\omega_{0})=0,$$
on $[0,\tau],$ with initial data $\omega_{0}.$ Moreover, it was also shown with the help of Proposition \ref{Mid-topo}, that there exists a $C^{s-4}$ vector field $\mathcal{R}:[0,\tau]\times \mathbbm{T}^{2}\to \mathbbm{T}^{2}$ such that
\begin{equation}
\label{DL}
V(t,x)+J\nabla\Delta^{-1}\omega_{0}(x)= t^{2} \mathcal{R}(t,x).
\end{equation}
Meanwhile, $\varphi_{F,t}(\omega_{0})$ satisfies the equation
$$\partial_{t} \varphi_{F,t}(\omega_{0}) - J\nabla \Delta^{-1}\omega_{0}\cdot \nabla \varphi_{F,t}(\omega_{0})=0.$$
Hence we have
$$\partial_{t}\left(\varphi_{F,t}(\omega_{0}) - \mathcal{S}_{t}(\omega_{0})\right) - J\nabla \Delta^{-1}\omega_{0}\cdot \nabla\left(\varphi_{F,t}(\omega_{0}) - \mathcal{S}_{t}(\omega_{0})\right) = \left(V+J\nabla\Delta^{-1}\omega_{0}\right)\cdot \nabla \mathcal{S}_{t}(\omega_{0}).$$
Therefore, using estimate \eqref{EE2} from Lemma \ref{stability}, applied to the vector field
$$X= J\nabla \Delta^{-1}\omega_{0},$$
that satisfies (using Proposition \ref{poisson-reg})
$$\left\| J\nabla \Delta^{-1}\omega_{0}\right\|_{H^{s-2}(\mathbbm{T}^{2})} \leq C\left\| \omega_{0}\right\|_{H^{s-3}(\mathbbm{T}^{2})},$$
we obtain the estimate
\begin{multline*}
\frac{\dd}{\dd t} \left\| \varphi_{F,t}(\omega_{0}) - \mathcal{S}_{t}(\omega_{0})\right\|_{H^{s-4}(\mathbbm{T}^{2})}^{2} \leq C\left\|\omega_{0} \right\|_{H^{s-3}(\mathbbm{T}^{2})} \left\| \varphi_{F,t}(\omega_{0}) - \mathcal{S}_{t}(\omega_{0})\right\|_{H^{s-4}(\mathbbm{T}^{2})}^{2} \\
+2 \left\| \varphi_{F,t}(\omega_{0}) - \mathcal{S}_{t}(\omega_{0})\right\|_{H^{s-4}(\mathbbm{T}^{2})} \left\| \left(V+J\nabla\Delta^{-1}\omega_{0}\right)\cdot \nabla \mathcal{S}_{t}(\omega_{0})\right\|_{H^{s-4}(\mathbbm{T}^{2})}.
\end{multline*}
Moreover, using identity \eqref{DL} and the $C^{s-4}$ regularity of $\mathcal{R},$ we have
$$\left\| \left(V+J\nabla\Delta^{-1}\omega_{0}\right)\cdot \nabla \mathcal{S}_{t}(\omega_{0})\right\|_{H^{s-4}(\mathbbm{T}^{2})} \leq C  t^{2} \left\| \mathcal{S}_{t}(\omega_{0})\right\|_{H^{s-3}(\mathbbm{T}^{2})}.$$
By Proposition \ref{numstab} and the hypothesis $\tau R_{0} \left\| \omega_{0}\right\|_{H^{s-3}(\mathbbm{T}^{2})}<1,$ we may write that
$$ \left\| \mathcal{S}_{t}(\omega_{0})\right\|_{H^{s-3}(\mathbbm{T}^{2})} \leq C \left\| \omega_{0}\right\|_{H^{s-3}(\mathbbm{T}^{2})}.$$
Collecting the previous estimates, we infer that
\begin{multline*}
\frac{\dd}{\dd t} \left\| \varphi_{F,t}(\omega_{0}) - \mathcal{S}_{t}(\omega_{0})\right\|_{H^{s-4}(\mathbbm{T}^{2})}^{2} \leq C\left\|\omega_{0} \right\|_{H^{s-3}(\mathbbm{T}^{2})} \left\| \varphi_{F,t}(\omega_{0}) - \mathcal{S}_{t}(\omega_{0})\right\|_{H^{s-4}(\mathbbm{T}^{2})}^{2} \\
+C \left\| \omega_{0}\right\|_{H^{s-3}(\mathbbm{T}^{2})} t^{2} \left\| \varphi_{F,t}(\omega_{0}) - \mathcal{S}_{t}(\omega_{0})\right\|_{H^{s-4}(\mathbbm{T}^{2})} .
\end{multline*}
Lemma \ref{inequadiff} gives us then the estimate
\begin{multline*}
\left\| \varphi_{F,t}(\omega_{0}) - \mathcal{S}_{t}(\omega_{0})\right\|_{H^{s-4}(\mathbbm{T}^{2})}\leq C\left\|\omega_{0}\right\|_{H^{s-3}(\mathbbm{T}^{2})}\int_{0}^{t}\left\| \varphi_{F,\sigma}(\omega_{0}) - \mathcal{S}_{\sigma}(\omega_{0})\right\|_{H^{s-4}(\mathbbm{T}^{2})}\dd \sigma \\
+C \tau^{3} \left\| \omega_{0}\right\|_{H^{s-3}(\mathbbm{T}^{2})}.
\end{multline*}
Therefore, if $\tau C \left\| \omega_{0}\right\|_{H^{s-3}(\mathbbm{T}^{2})}<1,$ which we may assume, choosing if necessary $R_{0} >C,$ we conclude that
$$\sup_{t\in [0,\tau]} \left\| \varphi_{F,t}(\omega_{0}) - \mathcal{S}_{t}(\omega_{0})\right\|_{H^{s-4}(\mathbbm{T}^{2})} \leq R_{1} \tau^{3} \left\| \omega_{0}\right\|_{H^{s-3}(\mathbbm{T}^{2})} ,$$
for some appropriate constant $R_{1}(C).$
\end{proof}

\subsection{An {\it a priori} global error estimate}

\begin{proposition}
\label{global}
Let $s\geq 5,$ and $\omega_{0} \in H^{s}(\mathbbm{T}^{2})$ with average $0.$ Let $\omega(t)\in C^{0}\left(\mathbbm{R}_{+},H^{s}(\mathbbm{T}^{2})\right)$ be the unique solution of equation \eqref{euler2d} with initial data $\omega_{0},$ given by Theorem \ref{existence}. For a time step $\tau\in ]0,1[,$ let $(\omega_{n})_{n\in\mathbbm{N}}$ be the sequence of functions starting from $\omega_{0}$ and defined by formula \eqref{omega-n} from iterations of the semi-discrete operator \eqref{SD-OP}. Assume that there exists a time $T_{0}>0$ and a constant $B>0$ such that
$$\sup_{t\in [0,T_{0}]}\left\| \omega(t)\right\|_{H^{s}(\mathbbm{T}^{2})} \leq B \quad \mbox{and} \quad \sup_{t_{n}\leq T_{0}} \left\| \omega_{n}\right\|_{H^{2}(\mathbbm{T}^{2})} \leq 2B.$$
Then there exists two positive constants $R_{0},R_{1}$ such that, if $\tau R_{0} B<1,$ 
$$\left\| \omega_{n}-\omega(t_{n})\right\|_{H^{s-4}(\mathbbm{T}^{2})} \leq \tau R(B)t_{n} e^{R_{1}T_{0}(1+B)},$$
for all $t_{n}\leq T_{0}+\tau,$ where $R:\mathbbm{R}_{+} \to \mathbbm{R}_{+}$ is an increasing continuous function that satisfies
$$R(B)\leq R_{1} \left(B + B^{3}\right).$$
\end{proposition}

\begin{proof}
We shall use the notations introduced previously
$$\omega_{n}=\mathcal{S}_{\tau}^{n}(\omega_{0}) \quad \mbox{and} \quad \omega(t_{n})=\varphi_{E,t_{n}}(\omega_{0}).$$
We may choose $R_{0}$ such that, if $\tau R_{0} B<1,$ Proposition \ref{numstab} holds when applied to any term of the sequence $\mathcal{S}_{\tau}^{n}(\omega_{0}),$ for $t_{n}\leq T_{0}.$ This implies that $\mathcal{S}_{\tau}^{n}(\omega_{0}) $ belongs to $H^{s}$ for all $t_{n}\leq T_{0}+\tau.$\\
That being said, the semi-discrete error is inductively expanded as follows
\begin{equation}
\notag
\begin{split}
\left\| \mathcal{S}_{\tau}^{n+1}(\omega_{0})-\varphi_{E,t_{n+1}}(\omega_{0})\right\|_{H^{s-4}(\mathbbm{T}^{2})} &\leq \left\|  \mathcal{S}_{\tau}\left(\mathcal{S}_{\tau}^{n}(\omega_{0})\right) - \mathcal{S}_{\tau}\left(\varphi_{E,t_{n}}(\omega_{0})\right)  \right\|_{H^{s-4}(\mathbbm{T}^{2})}\\
&+ \left\| \mathcal{S}_{\tau}\left(\varphi_{E,t_{n}}(\omega_{0})\right)- \varphi_{F,\tau}\left(\varphi_{E,t_{n}}(\omega_{0})\right) \right\|_{H^{s-4}(\mathbbm{T}^{2})}\\
&+  \left\| \varphi_{F,\tau}\left(\varphi_{E,t_{n}}(\omega_{0})\right) - \varphi_{E,\tau}\left(\varphi_{E,t_{n}}(\omega_{0})\right) \right\|_{H^{s-4}(\mathbbm{T}^{2})}, 
\end{split}
\end{equation}
for all $t_{n+1}\leq T_{0}+\tau.$\\
Applying Corollary \ref{numstab2}, and using the hypothesis $\tau R_{0}B<1,$ we may find a positive constant $R_{1}$ such that
$$\left\|  \mathcal{S}_{\tau}\left(\mathcal{S}_{\tau}^{n}(\omega_{0})\right) - \mathcal{S}_{\tau}\left(\varphi_{E,t_{n}}(\omega_{0})\right)  \right\|_{H^{s-4}(\mathbbm{T}^{2})}\leq e^{\tau R_{1} (1+B)} \left\| \mathcal{S}_{\tau}^{n}(\omega_{0})-\varphi_{E,t_{n}}(\omega_{0})\right\|_{H^{s-4}(\mathbbm{T}^{2})} + R_{1}B\tau^{3}.$$
Applying now Propositions \ref{error2} and \ref{error1}, we may moreover write that
$$\left\| \mathcal{S}_{\tau}\left(\varphi_{E,t_{n}}(\omega_{0})\right)- \varphi_{F,\tau}\left(\varphi_{E,t_{n}}(\omega_{0})\right) \right\|_{H^{s-4}(\mathbbm{T}^{2})}\leq R_{1}B\tau ^{3}$$
and
$$ \left\| \varphi_{F,\tau}\left(\varphi_{E,t_{n}}(\omega_{0})\right) - \varphi_{E,\tau}\left(\varphi_{E,t_{n}}(\omega_{0})\right) \right\|_{H^{s-4}(\mathbbm{T}^{2})} \leq R_{1}B^{3} \tau^{2}.$$
Therefore,
$$\left\| \mathcal{S}_{\tau}^{n+1}(\omega_{0})-\varphi_{E,t_{n+1}}(\omega_{0})\right\|_{H^{s-4}(\mathbbm{T}^{2})}  \leq e^{\tau R_{1} (1+B)} \left\| \mathcal{S}_{\tau}^{n}(\omega_{0})-\varphi_{E,t_{n}}(\omega_{0})\right\|_{H^{s-4}(\mathbbm{T}^{2})} + R(B)\tau^{2},$$
with
$$R(B)\leq R_{1}\left(B + B^{3}\right).$$
which implies by induction that for all $t_{n}\leq T_{0}+\tau,$
$$\left\| \mathcal{S}_{\tau}^{n}(\omega_{0})-\varphi_{E,t_{n}}(\omega_{0})\right\|_{H^{s-4}(\mathbbm{T}^{2})} \leq R(B)\tau^{2} \sum_{i=0}^{n-1} e^{t_{i}R_{1} (1+B)} \leq \tau R(B)t_{n} e^{R_{1}T_{0}(1+B)} ,$$
which concludes the proof.
\end{proof}

\subsection{Convergence of the semi-discrete scheme}

Here we shall prove our main result, namely Theorem \ref{convergence}. It will be a consequence of Proposition \ref{global}, and the only remaining task is to bootstrap controls of the same order on the $H^{s}$ norm of the exact solution and on the $H^{2}$ norm of the numerical solution up to a fixed time horizon.\\

\begin{Proofof}{Theorem \ref{convergence}}

Let $T$ and $B=B(T)$ be such that
$$\sup_{t\in [0,T]} \left\| \omega(t)\right\|_{H^{s}(\mathbbm{T}^{2})} \leq B.$$
We shall first prove by induction on $n\in \mathbbm{N},$ with $t_{n}\leq T,$ that
$$\sup_{t_{k}\leq t_{n}} \left\| \omega_{k} \right\|_{H^{2}(\mathbbm{T}^{2})}\leq 2B.$$
This clearly holds for $n=0.$ Let now $n\geq 1,$ with $t_{n}\leq T,$ and assume that the following induction hypothesis holds:
$$\sup_{t_{k}\leq t_{n}} \left\| \omega_{k} \right\|_{H^{2}(\mathbbm{T}^{2})}\leq 2B.$$
By applying Proposition \ref{global} with $s=6$ and $T_{0}=t_{n},$ we may find two positive constants 
$R_{0}, R_{1}$ such that, if $\tau R_{0} B<1,$ then for any $t_{k}\leq t_{n+1},$
$$\left\| \omega_{k} \right\|_{H^{2}(\mathbbm{T}^{2})} \leq \left\| \omega(t_{k}) \right\|_{H^{2}(\mathbbm{T}^{2})} + \left\| \omega_{k} - \omega(t_{k}) \right\|_{H^{2}(\mathbbm{T}^{2})} \leq B + \tau t_{k} e^{R_{1}T_{0}(1+B)}R(B),$$
with
$$R(B)\leq R_{1}\left(B + B^{3}\right).$$
If $\tau$ satisfies
$$\tau <\frac{B}{TR(B)e^{TR_{1}(1+B)}},$$
this yields
$$\left\| \omega_{k} \right\|_{H^{2}(\mathbbm{T}^{2})} \leq 2B,$$
for any $t_{k}\leq t_{n+1}.$ This concludes the induction and shows that
$$\sup_{t_{n}\leq T} \left\| \omega_{n} \right\|_{H^{2}(\mathbbm{T}^{2})}\leq 2B.$$
One obtains then the conclusion of Theorem \ref{convergence} by applying Proposition \ref{global} with $T_{0}=T-\tau.$

\end{Proofof}

\section{Appendix: Solving the Poisson equation}

\begin{proposition}
\label{poisson-eq}
Let $f\in L^{2}(\mathbbm{T}^{2}).$ Assume that 
$$\int_{\mathbbm{T}^{2}} f(x)\dd x =0.$$
Then there exists an unique $u\in H^{2}(\mathbbm{T}^{2})$ such that
\begin{equation}
\notag
\left\{
\begin{split}
&\Delta u =f\\
&\int_{\mathbbm{T}^{2}}u(x)\dd x=0.
\end{split}
\right.
\end{equation}
\end{proposition}

\begin{proof}
Assume first that $u\in H(\mathbbm{T}^{2})$ satisfies the equation. Since the average of $f$ on the torus is $0,$ we can write
$$f(x)=\sum_{k\in\mathbbm{Z}^{2*}} \hat{f}_{k} e^{ik\cdot x}.$$
Setting also
$$u(x)=\sum_{k\in\mathbbm{Z}^{2*}}\hat{u}_{k} e^{ik\cdot x}, $$
the Poisson equation simply reads
$$|k|^2\hat{u}_{k}=\hat{f}_{k}, \quad k\in\mathbbm{Z}^{2}.$$
Reciprocally, if the coefficients $\hat{u}_{k}$ are defined by the above formula (for $k\in \mathbbm{Z}^{2*}$), and if we set
$$u(x)=\sum_{k\in\mathbbm{Z}^{2*}}\hat{u}_{k} e^{ik\cdot x},$$
then $u\in H(\mathbbm{T}^{2})$ (this will be precisely proven in the next Proposition), has average $0,$ and satisfies $\Delta u=f.$
\end{proof}

\begin{proposition}
\label{poisson-reg}
Assume that $u$ and $f$ satisfy
$$\Delta u =f.$$
Then for all $s\geq2,$ 
\begin{equation}
\label{reg-L2}
\sup_{0\leq |\alpha| \leq s} \left\|\partial_{x}^{\alpha}u\right\|_{L^{2}(\mathbbm{T}^{2})}\leq C\left\| f\right\|_{H^{s-2}(\mathbbm{T}^{2})}.
\end{equation}
\end{proposition}

\begin{proof}
For any $0\leq |\alpha| \leq s,$ we have
\begin{multline*}
\left\| \partial_{x}^{\alpha}u\right\|_{L^{2}(\mathbbm{T}^{2})}^{2}=\sum_{k\in\mathbbm{Z}^{2}} |k|^{2\alpha} \left| \hat{u}_{k}\right|^{2}\leq C \sum_{k\in\mathbbm{Z}^{2}} |k|^{2\alpha-4}|k|^{2}  \left| \hat{u}_{k}\right|^{2} \\
\leq C\sum_{k\in\mathbbm{Z}^{2} }\langle k \rangle^{2s-4}\left| \hat{f}_{k}\right|^{2} =C\left\| f\right\|_{H^{s-2}(\mathbbm{T}^{2})}^{2}.
\end{multline*}
\end{proof}

\end{document}